\newif\ifpreprint
\preprinttrue 

\ifpreprint
\documentclass[12pt]{article}
\usepackage{fullpage}
\usepackage{version}

\title{Verification of First-Order Methods \\for Parametric Quadratic Optimization}
\author{ Vinit Ranjan \qquad Bartolomeo Stellato \\[1em] Department of Operations Research and Financial Engineering\\Princeton University}

\usepackage{amsmath,amssymb,amsthm,enumitem,mathtools}

\else

\RequirePackage{fix-cm}
\documentclass[smallextended]{svjour3}       \smartqed  

\journalname{Mathematical Programming}

\fi

\newcommand{\reviewChanges}[1]{{#1}}

\usepackage{graphicx}
\usepackage{pgfplots}
\pgfplotsset{compat=1.17}

\usepackage{multirow}
\usepackage{algorithm}\usepackage[noend]{algpseudocode} 

\usepackage{csvsimple}	\usepackage{booktabs}   \usepackage{adjustbox}  

\usepackage[colorlinks,
            linkcolor=cyan,
            citecolor=cyan,
            urlcolor=cyan,
            bookmarks]{hyperref}

\usepackage[acronyms]{glossaries}   \makeglossaries

\ifpreprint
\usepackage[textsize=tiny]{todonotes}

\else

\fi

\newcommand{\eg}{{\it e.g.}}
\newcommand{\ie}{{\it i.e.}}

\newcommand{\ones}{\mathbf 1}
\newcommand{\reals}{{\mbox{\bf R}}}

\newcommand{\symm}{{\mbox{\bf S}}}  
\newcommand{\NPhard}{\mbox{$\mathcal{NP}$-hard}}

\newcommand{\Tr}{\mathop{\bf tr}}
\newcommand{\diag}{\mathop{\bf diag}}

\providecommand{\argmin}{\mathop{\rm argmin}}

\newcommand{\trace}{\mathop{\bf tr}}

\newcommand{\paramset}{\Theta}

\newcommand{\prox}{{\bf prox}}

\newcommand{\Iff}{\quad \Longleftrightarrow \quad}

\ifpreprint
\newtheorem{theorem}{Theorem}[section]  \newtheorem{lemma}{Lemma}[section]    \newtheorem{proposition}{Proposition}[section]

\fi

\newacronym{PQP}{PQP}{parametric quadratic program}
\newacronym{PEP}{PEP}{performance estimation problem}
\newacronym{IQC}{IQC}{integral quadratic constraint}
\newacronym{LP}{LP}{linear program}
\newacronym{QP}{QP}{quadratic program}
\newacronym{SDP}{SDP}{semidefinite program}
\newacronym{QCQP}{QCQP}{quadratically constrained quadratic program}
\newacronym{RLT}{RLT}{reformulation-linearization technique}
\newacronym{DR}{DR}{Douglas-Rachford}

\newacronym{NNLS}{NNLS}{nonnegative least squares}
\newacronym{NUM}{NUM}{network utility maximization}
\newacronym{OSQP}{OSQP}{operator splitting quadratic programming}
\newacronym{ISTA}{ISTA}{iterative soft-thresholding algorithm}
\newacronym{FISTA}{FISTA}{fast iterative soft-thresholding algorithm}
\newacronym{LASSO}{LASSO}{least absolute shrinkage and selection operator}

\newacronym{MPC}{MPC}{model predictive control}
\newacronym{ADMM}{ADMM}{alternating direction method of multipliers}
\newacronym{FOM}{FOM}{first-order method}
\newacronym{CGAL}{CGAL}{conjugate gradient augmented Lagrangian}
\newacronym{SVM}{SVM}{support vector machine} 
\begin{document}

\ifpreprint \else
\title{Verification of First-Order Methods\\for Parametric Quadratic Optimization\thanks{The authors are supported by the NSF CAREER Award ECCS-2239771 and the Princeton School of Engineering and Applied Science Innovation grant.}
}

\author{Vinit Ranjan         \and
        Bartolomeo Stellato}

\institute{V. Ranjan \at
              Department of Operations Research and Financial Engineering, Princeton University\\
\email{vranjan@princeton.edu}           \and
           B. Stellato \at
           Department of Operations Research and Financial Engineering, Princeton University\\
           \email{bstellato@princeton.edu}
}

\date{Received: date / Accepted: date}
\fi

\maketitle

\ifpreprint
\let\oldtabular\tabular
\renewcommand{\tabular}{\small\oldtabular}
\fi

\begin{abstract}
We introduce a numerical framework to verify the finite step convergence of first-order methods for parametric convex quadratic optimization.
We formulate the verification problem as a mathematical optimization problem where we maximize a performance metric (\eg, fixed-point residual at the last iteration) subject to constraints representing proximal algorithm steps (\eg, linear system solutions, projections, or gradient steps). 
Our framework is highly modular because we encode a wide range of proximal algorithms as variations of two primitive steps: affine steps and element-wise maximum steps. 
Compared to standard convergence analysis and performance estimation techniques, we can explicitly quantify the effects of warm-starting by directly representing the sets where the initial iterates and parameters live.
We show that the verification problem is $\mathcal{NP}$-hard, and we construct strong semidefinite programming relaxations using various constraint tightening techniques.
Numerical examples in nonnegative least squares, network utility maximization, Lasso, and optimal control show a significant reduction in pessimism of our framework compared to standard worst-case convergence analysis techniques.
\ifpreprint \else
\keywords{First-order methods, Proximal algorithms, Parametric quadratic optimization, Convergence analysis, Semidefinite relaxations}
\fi
\end{abstract}

\section{Introduction}\label{sec:intro}
\subsection{Parametric quadratic programs}\label{sec:introPQP}
We consider convex \glspl{QP} of the form
\begin{equation}\label{prob:pqp}
	\begin{array}{ll}
		\text{minimize} & (1/2) x^T P x + q(\theta)^T x\\
		\text{subject to} & Ax \in \mathcal{C}(\theta),
	\end{array}
\end{equation}
with decision variable $x \in \reals^n$. 
In this work, we focus on the setting where we solve several instances of~\eqref{prob:pqp} where only parameters~$\theta \in \paramset \subset \reals^p$ vary.
The objective is defined by a symmetric positive semidefinite matrix $P \in \symm^{n\times n}_{+}$ and a linear term $q(\theta) \in \reals^n$ that depends on parameter~$\theta$.
The constraints are defined by a matrix $A \in \reals^{m \times n}$, and a nonempty, closed, convex set $\mathcal{C}(\theta) \subseteq \reals^n$ that depends on $\theta$.
In this paper, the set $\mathcal{C}(\theta)$ takes the form 
\begin{equation*}
\mathcal{C}(\theta) = [l(\theta),u(\theta)] = \left\{v \in \reals^m \mid l_i(\theta) \le v_i \le u_i(\theta)\right\},
\end{equation*}
with $l_i(\theta) \in \{-\infty\} \cup \reals$ and $u_i(\theta) \in \{\infty\} \cup \reals$ for all $\theta \in \Theta$.
We encode equality constraints in this form by setting $l_i(\theta) = u_i(\theta)$ for all $\theta \in \Theta$.
We refer to problem~\eqref{prob:pqp} as a {\it (convex) \gls{PQP}}.

\paragraph{Applications.}
A huge variety of applications from finance, engineering, operations research, and other fields involve the solution of parametric problems of the form~\eqref{prob:pqp}.
Financial applications include portfolio optimization, where we seek to rebalance asset allocation where parameters correspond to new returns estimates~\cite{markowitz,boyd2014portfolio}.
Applications in machine learning include hyperparameter tuning in \glspl{SVM}~\cite{svm}, Lasso~\cite{lasso,candes2007lasso} and Huber fitting~\cite{huber1964huber}.
Several problems in signal processing are also of the form~\eqref{prob:pqp}, where the parameters are the measured signals~\cite{sig_proc}.
Also in control engineering, \gls{MPC} problems involve the solution of \glspl{PQP} in real-time~\cite{borrelli2017mpc,rawlings2009mpc,marcucci2021ws_mpc}.

\paragraph{Solution algorithm.}
Our goal is to solve problem~\eqref{prob:pqp} using \emph{first-order methods} of the form
\begin{equation}\label{eqn:fom}
	z^{k+1} = T_\theta(z^k), \quad k = 0,1,\dots,
\end{equation}
where $T_\theta$ is a {\it fixed-point operator} that depends on the parameter $\theta$, and $z^k \in \reals^d$ is the iterate at iteration $k$.
The {\it fixed-points} $z^\star = T_\theta(z^\star)$ correspond to the optimal solutions of~\eqref{prob:pqp}.
To measure the distance to optimality, we define the {\it fixed-point residual} evaluated at $z^k$ as~$T_{\theta}(z^k) - z^k$, and measure its magnitude~$\|T_{\theta}(z^k) - z^k\|$.
As $k \to \infty$, the fixed-point residual converges to $0$ in well-known cases (\eg, averaged~\cite[Cor.\ 5.8]{ConvexAnalysisBausch2017} or contractive~\cite[Thm.\ 1.50]{ConvexAnalysisBausch2017}~\cite[Thm.\ 2.4.2]{LargeScaleConRyuE2022} operators).
In this paper, we study the behavior of the fixed-point operator $T_{\theta}$ without making any convergence assumption.

\paragraph{Verification problem.}
In several applications, we can only afford a specific number of iterations of algorithm~\eqref{eqn:fom} to compute an optimal solution for each instance of problem~\eqref{prob:pqp} in real-time~\cite{borrelli2017mpc,sig_proc}.
In this work, we numerically verify that the first-order method~\eqref{eqn:fom} is able to compute an approximate solution in $K$ iterations for any instance of problem~\eqref{prob:pqp}, by checking the condition
\begin{equation}
    \label{eqn:verify_problem}
	\|T_{\theta}^K(z^0) - T_{\theta}^{K-1}(z^0)\|^2 \le \epsilon, \quad \forall z^0 \in Z, \quad \forall \theta \in \Theta,
\end{equation}
where $\epsilon > 0$ is our tolerance and $T_{\theta}^K$ is the $K$-times composition of operator $T_{\theta}$.
We represent the set of initial iterates as $Z \subset \reals^d$ (for cold-started algorithms, $Z = \{ 0 \}$).
In this paper, we derive exact formulations and tractable approximations for this problem, highlighting its benefits compared to classical worst-case complexity verification techniques.

\subsection{Related work}

\paragraph{First-order methods.}
First-order optimization methods rely only on first-order derivative information~\cite{beck2017fom} and they were first applied to solve QPs in the 1950s~\cite{frank1956qp}.
Operator splitting algorithms~\cite{ConvexAnalysisBausch2017}, such as proximal algorithms~\cite{parikh2013prox} and the alternating direction method of multipliers (ADMM)~\cite{gabay1976admm,douglas1956dr,boyd2011admm}, are a particular class of first-order methods which solve an optimization problem by breaking it into simpler components (operators). 
Recently, first-order methods have gained wide popularity in parametric optimization and real-time applications for two reasons.
First, they involve computationally cheap iterations that are ideal for embedded processors with limited computing resources, such as those found in embedded control systems~\cite{megahertz,oscontrol,oscontrolreview}.
Second, they can be easily warm-started, which greatly reduces the real-time computational cost.
For these reasons, several general-purpose first-order optimization solvers are now available, including PDLP~\cite{pdlp} for \glspl{LP}, SCS~\cite{scs} for \glspl{SDP}, and the OSQP solver~\cite{stellato2020osqp} for \glspl{QP}.
Despite the latest advances, first-order methods are still highly sensitive to data and can exhibit very slow convergence for badly scaled problems, which can be dangerous for safety-critical applications.
In this project, given a family of PQPs, we will directly verify the performance of a specific first-order method for a fixed-number of iterations, while taking into account the effects of parameter variations and warm-starting.

\paragraph{Classical convergence analysis.}
Classical convergence analysis techniques for first-order methods in convex optimization consider general classes of problems, such as minimization of smooth or strongly convex functions~\cite[Chapter 2]{lecturesnesterov}\cite[Chapter 5]{beck2017fom}.
Some of these tools rely on operator theory by modeling the optimization problem in terms of finding a zero of the sum of monotone operators~\cite{LargeScaleConRyuE2022}. 
The monotonicity property is particularly general, \eg, the subdifferential of any convex proper function is a monotone operator~\cite[Section 2.2]{LargeScaleConRyuE2022}.
While often tight, these results can be quite pessimistic in the context considered in this work. 
First, they focus on generic classes of problems without targeting the structure of PQPs.
For example, in several applications such as linear \gls{MPC} or Lasso, the PQPs in equation~\eqref{prob:pqp} have a constant quadratic term, while other linear terms vary~\cite{stellato2020osqp}.
However, while all problems in a specific family have the same smoothness constant, first-order algorithms can exhibit very different practical performance. 
In fact, they can often converge more quickly than worst-case bounds~\cite[Section VII]{richter2012complexity}~\cite{stellato2020osqp}.
Second, classical convergence analysis techniques focus on asymptotic rates of convergence, which are particularly suited for a large number of algorithm steps, \ie, $K \to \infty$.
However, in real-time applications we have a finite budget of iterations which we cannot exceed. 
Compared to classical convergence analysis, this paper focuses on verifying that a first-order method provably achieves a high-quality solution for a specific family of PQPs, within a fixed number of iterations. 

\paragraph{Computer-assisted convergence analysis.} There are two primary approaches to automate the analysis of first-order methods for convex optimization, the {\it performance estimation problem (PEP)} approach, and the {\it integral quadratic constraints (IQCs)} approach. 
The PEP approach finds the worst-case finite sequence of iterates and problem instance to compute a performance guarantee for a given first-order method~\cite{drori2014pep,taylor2016smooth,taylor2017exactworst}.
Thanks to the easy-to-use PESTO~\cite{pesto2017} and PEPit~\cite{goujad2022pepit} toolboxes, PEP has led to many results in operator splitting algorithms~\cite{taylor2017exactworst,ryu2020os_pep}, distributed optimization~\cite{pep_distr} and the specific case of quadratic functions~\cite{bousselmi2022pepqp,bousselmi2023pep_linop}.
The IQC approach, instead, interprets the optimization algorithm as a dynamical system~\cite{lessard2016IQC,taylor2018lyapunov,iqc_nesterov} and builds sufficient conditions for convergence (stability) via integral quadratic constraints.
Both PEP and IQC approaches solve a \gls{SDP} to compute the worst-case bounds. 
In particular, PEP provides guarantees for a finite number of steps, while the IQC approach provides infinite-horizon bounds by solving a smaller problem that is independent from the number of steps considered.
In addition, the IQC approach can only analyze linearly convergent algorithms, while the PEP approach can deal with sublinearly convergent methods.
The PEP approach has been extended to a nonconvex formulation that can be used to generate theoretically optimal first order methods via custom branch-and-bound procedures \cite{dasgupta2022BnBPEP,jang2023optista}.
In this work, we also solve \glspl{SDP} to analyze the performance of first-order methods for a finite number of steps, but we focus on the geometric properties of the iterates for a specific parametric family of problems.
In Section~\ref{sec:PEPdiff}, we highlight the differences between our framework and PEP, explaining the main advantages and disadvantages.

\paragraph{Neural network verification.}
Our work draws inspiration from neural network verification techniques, which aim to certify that a trained network produces the correct output for any acceptable input (\eg, perturbation of the training data)~\cite{intro_nnverify}.
These approaches formulate verification as an optimization problem where the objective encodes the performance quality (\eg, distance between network output and classification threshold) while the constraints represent the input propagation across the network layers.
Solving such problems can be interpreted as searching for adversarial examples that challenge the network behavior.
Several recent techniques address this problem by either global optimization tools, heuristics, or convex relaxations~\cite{liu2021algorithms}.
Global techniques (also called, complete verifiers) encode the verification problem as a mixed-integer linear program (MILP)~\cite{mipnnfischetti,mipverify,bunel2020branch} where the discrete decisions correspond to identifying specific regions of piecewise (\eg, ReLU) activation funtions.
Thanks to recent developments in efficient bound propagation~\cite{zhang2018efficient,wang2021beta}, custom cutting planes~\cite{zhang2022general}, and GPU-based bound tightening~\cite{xu2021fast} and branch-and-bound algorithm~\cite{zhang22babattack}, global optimization techniques were able to scale to large neural networks.
However, all such algorithms ultimately rely on exhaustive search which can be computationally challenging.
Several relaxed formulations seek to overcome this challenge using linear programming~\cite{dathathri2020enabling,deepsplit}, and semidefinite programming~\cite{fazlyabsdp,brown2022unifyingNN,raghunathan2018semidefinite,sparsennverify}.
By analyzing finite step algorithms as fixed-depth computational graphs, our approach is closely related to neural network verification.
In a very similar spirit, we search for worst-case inputs of our algorithm (\ie, problem parameters and initial iterates) to verify that a convergence criterion is within the desired threshold.
In particular, we show that several operators in neural network graphs, such as ReLU activation functions and feedforward linear transformation, correspond to steps of proximal operators for solving QPs.

\paragraph{Generalization bounds in learned optimizers.}
Popular machine learning approaches, such a {\it learning to optimize (L2O)}~\cite{chen2022l2o} or {\it amortized optimization}~\cite{amos_tutorial}, seek to improve the performance of optimization algorithms by learning algorithm steps for a specific distribution of problem instances.
Unfortunately, verifying the performance of such learned algorithms is still an open challenge.
By considering $K$ steps of a first-order method as a computational graph, Sambharya et al.~\cite{sambharya2022l2ws,sambharya2023learning} combine PAC-Bayes~\cite{mcallester1998pacbayes,shawe-taylor1997pacbayes} and monotone operator~\cite{ConvexAnalysisBausch2017} theories to provide probabilistic generalization bounds on the fixed-point residual.
While also analyzing optimization algorithms as fixed-depth computational graphs, in this work we focus on deterministic iterations without learned components and we model the distribution of problems as a parametric family where parameters fall in a predefined convex set.

\subsection{Contributions}
In this paper, we present a framework to verify the performance of first-order methods for parametric quadratic optimization.
Our contributions are as follows:
\begin{itemize}
    \item We define a verification problem to analyze the performance of first-order methods for parametric quadratic optimization as a mathematical optimization problem. We encode a variety of proximal algorithms as combinations of two {\it primitive steps}: affine steps and element-wise maximum steps. We also explicitly quantify the effects of warm-starting by directly representing the sets where the initial iterates and parameters live.
    \item We show that for unconstrained \glspl{PQP} the verification problem is a convex semidefinite program, and we provide illustrative examples to compare our formulation to the performance estimation approach. For constrained \glspl{PQP}, we show that the verification problem is \NPhard.
    \item We construct strong convex \gls{SDP} relaxations for constrained \glspl{PQP} based on the primitive steps and the sets where parameters and initial iterates live. We also strengthen the relaxation with various bound propagation and constraint tightening techniques.
    \item We compare the performance of our approach with standard convergence analysis and performance estimation tools, obtaining less pessimistic worst-case bounds in various numerical examples from nonnegative least-squares, network utility maximization, Lasso, and optimal control. We also show \gls{PQP} instances where our approach can uncover details about practical convergence behavior that the standard analysis does not. The code to reproduce these results is available at \url{https://github.com/stellatogrp/sdp_algo_verify}.
\end{itemize}

\section{Performance verification}
From problem~\eqref{eqn:verify_problem}, our goal is to verify that all realizations of the parametric family lead to an iterate sequence with small fixed-point residual.
This is equivalent to checking that the optimal objective of the following optimization problem is less than a threshold $\epsilon>0$:
\begin{equation}\label{prob:general_verifyprob}
\tag{VP}
\begin{array}{ll}
    \mbox{maximize} & \|z^K - z^{K-1}\|^2 \\
    \mbox{subject to} & z^{k+1} = T_{\theta}(z^k),\quad k=0,\dots,K-1\\
    & z^0 \in Z, \quad\theta \in \Theta,
\end{array}
\end{equation}
with variables being the problem parameter $\theta \in \reals^p$ and the iterates $z^0,\dots,z^{K} \in \reals^d$.
While a fixed-point iteration can be written as a single operator $T_{\theta}$, often times $T_{\theta}$ is composed of $l$ different operators, \ie,
\begin{equation}\label{eqn:operator_chain}
    z^{k+1} = T_{\theta}(z^k) = (S_{\theta}^1 \circ S_{\theta}^2 \circ \dots \circ S_{\theta}^l)(z^k).
\end{equation}
We discuss specific choices of $S^j_\theta$ for $j=1,\dots,l$ in Section~\ref{sec:fixedpt_core} and~\ref{sec:fixedpt_extrasteps}.
In the following, we simplify the notation by dropping index $j$, and expressing each operator as $$y = S_{\theta}(x),$$ where $x \in \reals^d$ is its input and $y \in \reals^d$ its output. 
Lastly, we discuss the initial sets~$Z$ and the parameter sets~$\Theta$ in Section~\ref{sec:verifyprob_initsets}.

\subsection{Primitive algorithm steps}\label{sec:fixedpt_core}
We incorporate fixed-point iterations as constraints to link successive iterates. 
We present two fundamental operators that serve as building blocks for more complex iterations.

\paragraph{Affine.}
Consider the {\it affine} operator $y = S_\theta(x)$ representing the solution of the following linear system 
\begin{equation}\label{eqn:linearupdate}
	Dy = Ax + B q(\theta),
\end{equation}
where $A$ and $D$ are square matrices in $\reals^{d \times d}$ and $B \in \reals^{d \times n}$.
We assume that $D$ is invertible, which makes $y$ the unique solution to a linear system and therefore a well defined iteration.
Affine steps can have varying coefficients per iteration, \ie, $A^k, B^k, D^k$ for $k=0,\dots, K-1$.

 \paragraph{Element-wise maximum.}
Consider the {\it element-wise maximum} of $x$ and $l(\theta)$, \ie, $y = S_\theta(x) = \max\{x, l(\theta) \}$.
We can represent this operator with the following conditions~\cite[Section III.D]{fazlyabsdp}\cite[Section 2.2]{brown2022unifyingNN}:
\begin{equation}\label{eqn:generalmaxconstraints}
    y = \max\{x, l(\theta) \} \Iff y \geq l(\theta), \quad y \geq x, \quad (y - l(\theta))^T(y - x) = 0.
\end{equation}

\subsection{Fixed-point operators}\label{sec:fixedpt_extrasteps}
Using the primitive steps above, we now define operators representing fixed-point algorithm steps.

\paragraph{Gradient steps.} Consider a gradient step with step size $t$ mapping iterate $x$ to $y$,
$$y = x - t \nabla f_\theta(x) = (I - tP)x - tq(\theta).$$
Here, in the second equality we used the gradient of a quadratic function $\nabla f_{\theta}(x) = Px + q(\theta)$.
We can write this step as an affine step in equation~\eqref{eqn:linearupdate} with $D = I$, $A = I - tP$, and $B = -tI$.
This formulation can also represent varying step sizes $t_k$, by havng $A^k = I - t_kP$ and $B^k = -t_kI$.

\paragraph{Momentum steps.} 
Momentum steps such as Nesterov's acceleration are commonly used to improve convergence of first order methods~\cite{nesterov1983amf,beck2009ista}.
We can write a Nesterov's accelerated gradient step as the combination of two steps mapping iterates $(r, w)$ to $(\tilde{r}, \tilde{w})$,
\begin{equation*}
\begin{array}{ll}
    \tilde{w} &= r - t \nabla f_{\theta}(r)\\
    \tilde{r} &= (1 + \beta)\tilde{w} - \beta w.
\end{array}
\end{equation*}
By substituting the gradient $\nabla f_{\theta}(x) = Px + q(\theta)$, and plugging the first step into the second one, we can write
\begin{equation*}
\begin{array}{ll}
    \tilde{w} &= (I - tP) r - t q(\theta)\\
    \tilde{r} &= - \beta w + (1 + \beta)(I - tP) r - t(1 + \beta) q(\theta).
\end{array}
\end{equation*}
This step corresponds to an affine step in equation~\eqref{eqn:linearupdate} with $y = (\tilde{w}, \tilde{r})$, $x = (w, r),$ and
\begin{equation*}
    D = I, \quad A = \begin{bmatrix}
        0 & I-tP \\
        -\beta I & (1+\beta)(I - tP)
    \end{bmatrix}, \quad
    B = \begin{bmatrix}
        -tI \\
        -t(1+\beta)I
    \end{bmatrix}.
\end{equation*}
In some variants, parameters $\beta$ are updated at each iteration using predefined rule, \eg, $\beta^{k} = (k-1)/(k+2)$~\cite{nesterov1983amf}.
We can represent such cases in our framework by having varying matrices $A^k,$ and $B^k$.

\paragraph{Proximal steps.}
We model several proximal steps with are the building blocks of commonly used proximal algorithms~\cite[Section 1.1]{parikh2013prox}.
For a function $f : \reals^d \to \reals$, the proximal operator $\prox_{f}: \reals^d \to \reals^n$ is defined by:
\begin{equation*}
	y = \prox_{f}(x) = \argmin_v \left(f(v) + (1/2) \left\| v - x \right\|_2^2 \right).
\end{equation*}
We focus on the following proximal operators, summarized in Table~\ref{tab:proxoper_constraints}:
\begin{itemize}
    \item  {\bf $\ell_1$-norm.} The proximal operator of a function $f(x) = \lambda \| x \|_1$ is the (element-wise) \textit{soft-thresholding} operator~\cite[Section 6.5.2]{parikh2013prox}, 
    $$y = \prox_f(x) = (x - \lambda)_+ - (-x - \lambda)_+ = \max\{x, \lambda\} - \max\{-x, \lambda\},$$
    where $(v)_+ = \max\{v, 0\}$. The soft-thresholding operator is, therefore, the composition of two elementwise maximum steps with an affine subtraction step. 
\item {\bf Quadratic function.} The proximal operator of a convex quadratic function $f(x) = (1/2)x^T P x + q(\theta)^Tx + c$, with $P \in \symm_+^n$, is~\cite[Section 6.1.1]{parikh2013prox} 
    $$y = \prox_f(x) = (I+P)^{-1}(x - q(\theta)).$$
Since positive semidefiniteness of $P$ implies invertibility of $I+P$, this corresponds to an affine step with $D = I + P$, $A = I$, and $B = -I$.
\item {\bf Indicator function of a box.} 
Consider the indicator function of a box $C(\theta) = [l(\theta), u(\theta)]$ defined as $f(x) = \mathcal{I}_{[l(\theta), u(\theta)]}(x) = 0$ if $l(\theta) \le x \le u(\theta)$ and $\infty$ otherwise.
Its proximal operator is the projection operator~\cite[Section 1.2]{parikh2013prox} 
$$y = \prox_{f}(x) = \max\{\min\{x, u(\theta)\}, l(\theta)\} = \max\{-\max\{-x, -u(\theta)\}, l(\theta)\},$$
which corresponds to the composition of two element-wise maximum steps.
\end{itemize}

\subsection{Initial iterate and parameter sets}\label{sec:verifyprob_initsets}
We impose constraints on the initial iterates $z^0 \in Z$, and the parameters $\theta \in \Theta$, where $Z$ and $\Theta$ are convex sets of the following forms.

\paragraph{Hypercubes.}
We can represent hypercubes of the form $\{v \in \reals^d \mid l \leq v \leq u \}$, with element-wise inequalities and $l_i \in \{-\infty\}\cup \reals$ and $u_i \in \{\infty\}\cup\reals$.

\paragraph{Polyhedra.}
We represent polyhedral constraints of the form $\{ v \in \reals^d \mid A v \leq b \},$ with element-wise inequalities defined by $A \in \reals^{\ell \times d}$ and $b \in \reals^{\ell}$.
\reviewChanges{
By introducing a slack variable $s \in \reals^\ell$, the set can be embedded into a higher dimensional space as $\{ (v, s) \in \reals^d \times \reals^\ell \mid A v + s = b, \; s \ge 0 \}$.
}

\paragraph{$\ell_p$-balls.}
For $p = 1, 2, \infty$, we consider $\ell_p$ balls centered at $c$ with radius $r$, that is $\{v \in \reals^d \mid \| v - c \|_p \leq r \}.$ 
We define $\ell_1$ balls using linear inequalities and $\ell_2$ balls with quadratic inequalities.  Lastly, $\ell_\infty$ balls are special cases of hypercubes, which we treat separately.

\section{Differences from performance estimation problems}\label{sec:PEPdiff}
In this section, we highlight specific differences from the PEP framework.
One attractive property of the PEP framework is its dimension-free \glspl{SDP} formulations. That is, the size of the variables in a PEP \gls{SDP} scales linearly in $K$ and does not depend on the dimension of the iterates.
Our verification problem~\eqref{prob:general_verifyprob}, instead, is not dimension-free and is more computationally intensive to solve.
However, we observe the benefits of much tighter bound analysis for \glspl{PQP}.
We showcase some comparison examples for unconstrained \glspl{PQP} in Section~\ref{subsec:unconstrainedQP}.

\paragraph{Functional representation.}
Since the first seminal papers~\cite{drori2014pep,taylor2016smooth}, the PEP framework considers the class of $L$-smooth and $\mu$-strongly convex functions.
To represent such infinite-dimensional objects using tractable~\gls{SDP} formulations, PEP uses {\it interpolation constraints} representing their effects on finite-dimensional quantities of interest (\eg, gradients, functions values).
Recent works derived necessary and sufficient interpolation constraints specifically designed for \textit{homogeneous} quadratic functions~\cite{bousselmi2022pepqp,bousselmi2023pep_linop}.
However, while tighter than generic interpolation constraints, such inequalities are no longer sufficient for nonhomogeneous quadratic functions~\cite[Section 3.4]{bousselmi2023pep_linop}.

In this work, we consider a specific class of parametric quadratic functions where the quadratic term is constant.  
Rather than using interpolation inequalities, we directly represent the algorithm steps as affine steps or element-wise maximum steps.
Overall, our functional representation can be seen as an explicit characterization of gradient and proximal information as opposed to the implicit characterization in PEP.

\paragraph{Parametrization and functional shifts.}
Consider two different \gls{PQP} instances indexed by $\theta_1$ and $\theta_2$, \ie, and with the same matrix $P$ but different linear terms $q(\theta_1)$ and $q(\theta_2)$.
Since $P$ is the same, both \glspl{PQP} have the same smoothness and strong convexity properties, so a PEP SDP can only differentiate between these two by altering the initial condition~$\left\| z^0 - z^\star(\theta) \right\|^2 \leq R^2$, where the optimal solution is~$z^\star(\theta) = P^{-1}q(\theta)$.
Therefore, for a given initial iterate $z^0$, $R$ must be large enough to consider the worst-case distance $\max\{\left\| z^0 - z^\star(\theta_1) \right\|,\left\| z^0 - z^\star(\theta_2) \right\|\}$, which can lead to a large amount of conservatism.
In contrast, our explicit characterization of how $\theta$ enters into the problem allows for a tighter bound analysis.

By directly considering the function curvature, the PEP problem is invariant under any orthogonal transformation of the iterates~\cite[Section 3.1]{drori2014pep}.
In addition, the classes of strongly convex and smooth functions~\cite[Section 3]{taylor2016smooth} and homogeneous quadratic functions~\cite[Section 4]{bousselmi2022pepqp} are invariant under additive shifts in the domain.
In this way, the optimal solution $x^\star$ can be shifted to 0 without loss of generality.
While these conditions mostly serve to simplify the PEP formulation, 
in the practical settings considered in this work, we have access to the explicit form of $P$ and $q(\theta)$ as opposed to more general function classes.
By dropping the dimension-free property for our verification problems, we are able to encode the specific \gls{PQP} structure.

\paragraph{Gram matrix and the initial set.}
The PEP \glspl{SDP} use a Gram matrix formulation to represent the interpolation conditions in terms of function values and inner products between iterates and gradients.
However, this means when solving a PEP SDP, we optimize over the inner product values and we do not have access to the vectors themselves.
So, the PEP framework cannot express specific $\ell_2$-ball or box constraints to analyze warm-starting.
For example, consider a constraint of the form 
\begin{equation*}
 \left\| z^0 - v \right\| \leq 1 \quad \iff \quad (z^0)^T z^0 - 2 v^T z^0 + v^T v \leq 1,
\end{equation*}
where $v$ is a given vector.
Since the Gram matrix can only represent inner products by their value, it cannot uniquely identify the components of vectors $v$ and $z^0$.
In contrast, the vector characterization in our verification problem allows to explicitly represent this kind of constraints in terms of $v$.

\paragraph{Relative vs. absolute bounds.}
The performance estimation problem is a flexible framework to compute relative convergence bounds with respect to the initial distance to optimality.
Specifically, PEP computes the smallest constant $\tau$ (\ie, contraction factor) that bounds the optimality gap after~$K$ iterations using the initial distance from the optimal solution (see~\cite[Section 2.1]{goujad2022pepit}),
\begin{equation}\label{eqn:pep_tau}
    \lVert z^K - z^\star \rVert^2 \leq \tau\lVert z^0 - z^\star \rVert^2.
\end{equation}
These bounds are relative because they rely on the initial distance to optimality, $\lVert z^0 - z^\star \rVert^2$.
In contrast, our method directly computes the worst-case bound without using any information on the optimal solution $z^\star$ (\eg, initial distance to optimality), and, therefore, it provides absolute convergence bounds.

\subsection{Unconstrained QP}\label{subsec:unconstrainedQP}
Consider an unconstrained~\gls{PQP} with~$q(\theta) = \theta$ of the form
\begin{equation}\label{prob:unconstrained_QP}
     \begin{array}{ll}
		\text{minimize} & (1/2)x^T P x + \theta^T x,
    \end{array}
\end{equation}
where $P$ is a positive definite matrix.
The optimal solution is $x^\star(\theta)=-P^{-1} \theta$.
We solve this problem using gradient descent with step size~$t$, whose the fixed-point iterations are
\begin{equation*}
    z^{k+1} = z^k - t (P z^k + \theta),
\end{equation*}
which we rewrite in terms of $z^0$ as
\begin{equation}\label{eqn:uqp_gd_zk}
    z^{k+1} = (I-tP)^{k+1} z^0 - \sum_{i=0}^{k} (I-tP)^i (t\theta).
\end{equation}
It is well known that for $L$-smooth and $\mu$-strongly convex functions, we have
\begin{equation*}
    \lVert z^k - z^\star \rVert \leq \tau^k \lVert z^{0} - z^\star \rVert,
\end{equation*}
where  $\tau \in (0, 1)$ as long as $t \in (0, 2/L)$ and the worst-case convergence rate~$\tau$ is minimized when $t = 2/(\mu + L)$~\cite[Chapter 9]{boyd2004convex}\cite[Chapter 2]{lecturesnesterov}.
When replacing the distance to optimality $\left\| z^0 - z^\star\right\|$ with the fixed-point residual $\left\|z^{k} - z^{k-1}\right\|$, the same value of $t$ minimizes the upper bound, 
as shown in Proposition~\ref{prop:optimal_t}, whose proof is delayed to Appendix~\ref{apx:opt_t_proof}.
\begin{proposition}\label{prop:optimal_t}
    Let $0 < \mu \leq L$, $t> 0$, and $\tau = \max\{|1-t\mu|, |1-tL|\}$.
    For any $L$-smooth and $\mu$-strongly convex function and gradient descent with fixed step size $t$, we have
    \begin{equation*}
        \left\| z^{k} - z^{k-1}\right\| \leq \tau^{k-1}(1 + \tau) \left\| z^{0} - z^\star \right\|,
    \end{equation*}
    and the worst-case upper bound is minimized for~$t = 2/(\mu + L)$.
\end{proposition}

To form the verification problem and analyze~\eqref{eqn:uqp_gd_zk}, we define 
\begin{equation*}
    H = (I-tP)^K - (I-tP)^{K-1}, \quad E = -t(I-tP)^{K-1}.
\end{equation*}
The verification problem becomes
\begin{equation}\label{prob:UQP_verifyprob}
     \begin{array}{ll}
		\text{maximize} & \lVert Hz^0 + E\theta\rVert^2 \\
            \text{subject to} & z^0 \in Z, \quad \theta \in \Theta.
    \end{array}
\end{equation}
If $Z$ and $\Theta$ are representable with quadratic constraints (\eg, ellipsoidal sets), problem~\eqref{prob:UQP_verifyprob} is a nonconvex~\gls{QCQP}.
For this section, let the initial iterate set be the $\ell_2$-ball $Z = \{z \in \reals^d \mid \|z - c_z\|_2 \le r_z\}$ with center $c_z \in \reals^d$ and radius $r_z \geq 0$, and the parameter set be the $\ell_2$-ball $\Theta = \left\{\theta \in \reals^p \mid \|\theta - c_{\theta}\|_2 \le r_{\theta}\right\}$ with center $c_{\theta} \in \reals^p$ and radius $r_{\theta} \geq 0$.
With these sets, we rewrite problem~\eqref{prob:UQP_verifyprob} as:
\begin{equation}\label{prob:UQP_VP_l2ballsets}
    \begin{array}{ll}
		\text{maximize} & \lVert Hz^0 + E\theta\rVert^2 \\
            \text{subject to} & \|z^0 - c_z\|_2 \le r_z, \quad \|\theta - c_{\theta}\|_2 \le r_{\theta}.
    \end{array}
\end{equation}
To derive a convex problem, we lift the problem into a higher dimension and construct an \gls{SDP} reformulation~\cite[Section 3.3]{park2017general} \cite[Section 2]{wang2021sdpqcqp}.
When there is only one quadratic constraint, the S-Lemma guarantees tightness of the \gls{SDP} relaxations~\cite{yakubovich1971slemma,polik2007slemma}.
We now showcase three examples with only one quadratic constraint so that the S-Lemma holds and we can reformulate the verification problem as a convex \gls{SDP}.
In particular, for each of these experiments we will set either $r_z = 0$ or $r_\theta =0$ to reduce either $Z$ or $\Theta$, respectively, to be a singleton.
For consistency, we choose $t=2/(\mu + L)$ across all examples and we compare our results with the solutions using the PEP framework through the PEPit toolbox \cite{goujad2022pepit}.
It is worth noting that we solve the PEP \glspl{SDP} over the class of $\mu$-strongly convex and $L$-smooth {\it quadratic} functions for the most relevant comparison \cite{bousselmi2022pepqp,bousselmi2023pep_linop}.

\paragraph{Initial iterate set comparison.}
First, we isolate the effect of the initial iterate sets only, by setting $r_\theta = 0$.
In this case, we can rewrite the verification problem~\eqref{prob:UQP_VP_l2ballsets} as
\begin{equation}\label{prob:UQP_VP_Zonly}
    \begin{array}{ll}
		\text{maximize} & \lVert Hz^0 + Ec_\theta \rVert^2 \\
            \text{subject to} & \|z^0 - c_z\|_2 \le r_z.
    \end{array}
\end{equation}
By squaring the constraints and expanding the objective, we can rewrite problem~\eqref{prob:UQP_VP_Zonly} as
\begin{equation}\label{prob:UQP_VP_Zonly_expanded}
    \begin{array}{ll}
		\text{maximize} & (z^0)^T H^T H z^0 + 2 (E c_\theta)^T H z^0 + (E c_\theta)^TE c_\theta \\
            \text{subject to} & (z^0)^Tz^0 - 2c_z^T z^0 \le r_z^2 - c_z^Tc_z.
    \end{array}
\end{equation}
To get a convex formulation, given a general quadratic form $v^TJv$ with $v \in \reals^d$, we can write
\begin{equation*}
    v^T J v = \trace(v^T J v) = \trace(J vv^T) = \trace(JM),
\end{equation*}
where the second equality follows from the cyclic property of the trace, and in the last equality we introduce a matrix variable $M = vv^T$.
We relax this rank-one constraint to an inequality constraint $M \succeq vv^T$ and note that
\begin{equation}\label{eqn:matvar_psdrelaxation}
    M \succeq vv^T \Iff \begin{bmatrix}
        M & v \\
        v^T & 1 
    \end{bmatrix} \succeq 0,
\end{equation}
where the equivalence holds due to the Schur complement~\cite[Section B.2]{boyd2004convex}.
With this technique, we derive our convex reformulation of problem~\eqref{prob:UQP_VP_Zonly} as the following \gls{SDP}:
\begin{equation}\label{prob:UQP_VP_Zonly_SDP}
    \begin{array}{ll}
		\text{maximize} & \trace(H^T H M) + 2 (E c_\theta)^T H z^0 + (E c_\theta)^TE c_\theta \\
            \text{subject to} & \trace(M) - 2c_z^T z^0 \le r_z^2 - c_z^Tc_z \\
            & \begin{bmatrix}
                M & z^0 \\
                (z^0)^T & 1 
            \end{bmatrix} \succeq 0.
    \end{array}
\end{equation}
To show the effect of the initial iterates set, we set $\Theta = \{0\}$ and, therefore, the optimal solution $x^\star = 0$.
For simplicity, let $P = \diag(\mu, L)$
with $\mu =1, L=10$.
We first solve the verification problem SDP~\eqref{prob:UQP_VP_Zonly_SDP} with initial set $\overline{Z} = \left\{z \in \reals^2 \mid \|z\|\le 1\right\}$ (\ie~$c_z = 0$, $r_Z = 1$) and let $\overline{z}$ be the worst-case initial iterate.
Then, we solve the verification problem \gls{SDP} with three smaller initial sets.
First, we have $Z_1 = \left\{z \in \reals^2 \mid \left\|z - 0.9 \overline{z}\right\| \le 0.1\right\}$, which is explicitly constructed to contain the worst-case iterate for the larger initial set $\overline{Z}$.
Then, we set $Z_2 = \left\{z \in \reals^2 \mid \left\|z - c_2\right\| \le 0.1\right\}$ with $c_2 = (0.842, -0.317)$, and $Z_3 = \left\{z \in \reals^2 \mid \left\|z - c_3\right\| \le 0.1\right\}$ with $c_3 = (-0.397, -0.807)$.
Note that $c_2$ and $c_3$ are two randomly generated vectors with $2$-norm equal to $0.9$.
Lastly, we also compare with the worst-case iterate from the PEP framework with initial condition $\left\| z^0 - z^\star \right\|^2 \leq 1.$
Results are shown in Figure~\ref{fig:uqp_exp1}.
\begin{figure}
\begin{minipage}{\textwidth}
    \centering
    \raisebox{-1em}{\includegraphics[width=0.44\textwidth]{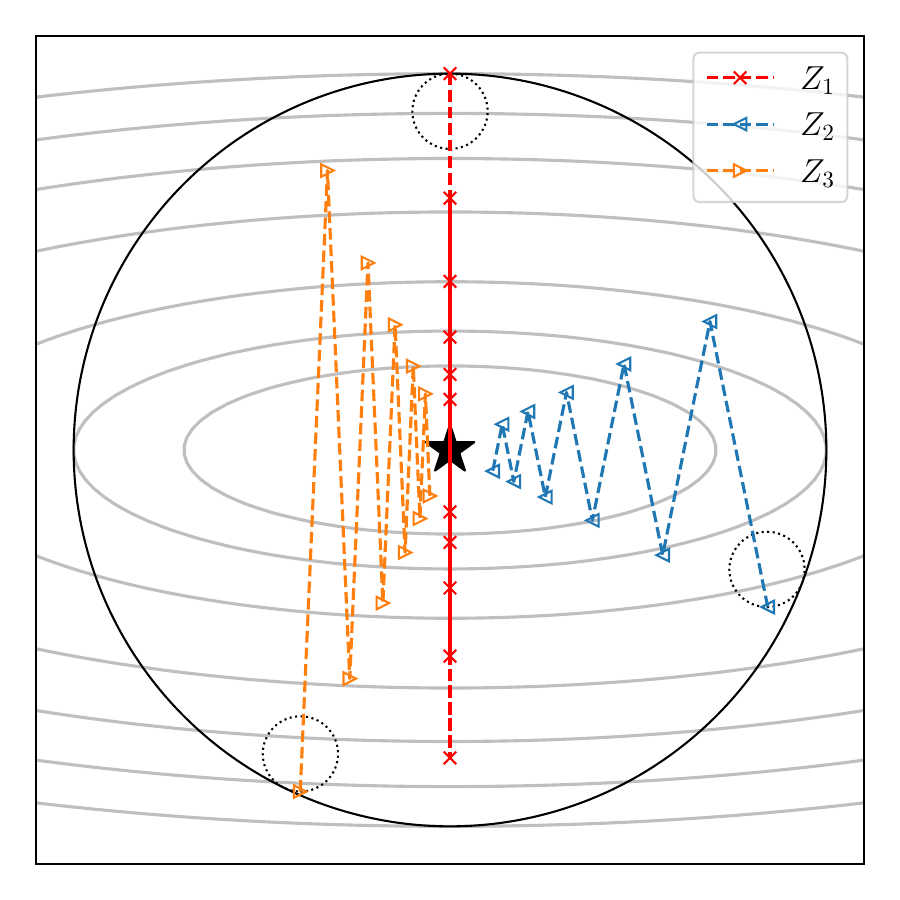}}
    \includegraphics[width=0.48\textwidth]{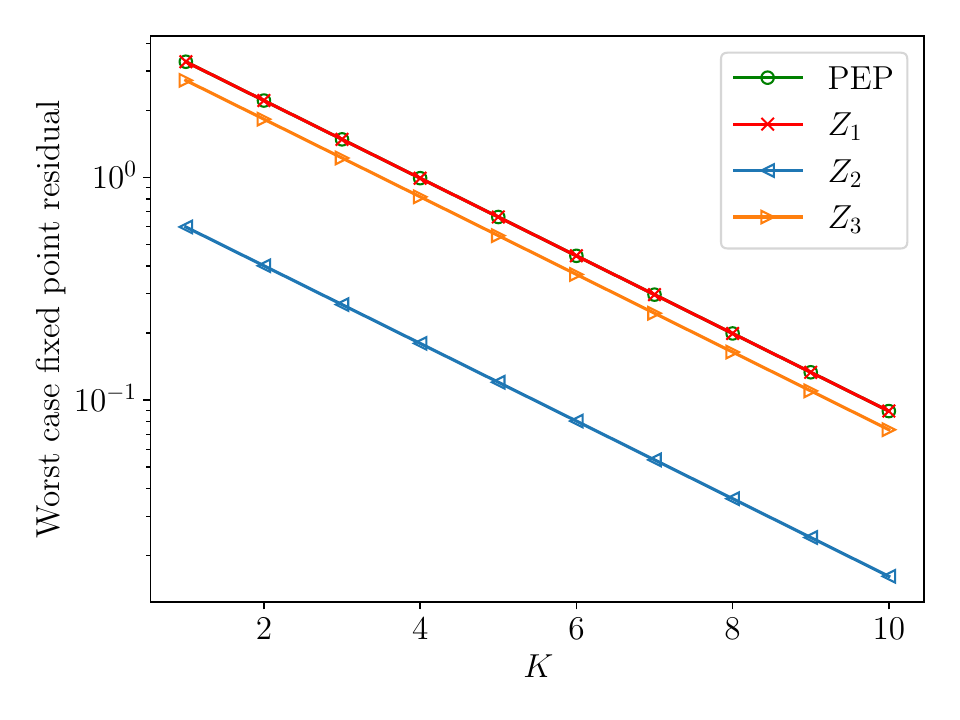}
\end{minipage}
    \caption{Unconstrained QP experiment with different the initial iterate sets (dashed circles). On the left, we show the iterate paths, and on the right, we plot the worst cast fixed-point residuals for up to $K=10$ steps. We also solve the PEP SDP with initial distance to optimality bounded by 1 (solid black circle). The verification problem and PEP agree on the worst-case value for $Z_1$.}
    \label{fig:uqp_exp1}
\end{figure}

\paragraph{Parameter set comparison.}
For this example, we isolate the effect of the parameter sets only, by setting $r_Z = 0$.
We follow a similar derivation as in the initial iterate set comparison experiment and arrive at the following verification problem~\gls{SDP}:
\begin{equation}\label{prob:UQP_VP_Thetaonly_SDP}
    \begin{array}{ll}
		\text{maximize} & \trace(E^T E M) + 2 (H c_z)^T E \theta + (H c_z)^TH c_z \\
            \text{subject to} & \trace(M) - 2c_\theta^T \theta \le r_\theta^2 - c_\theta^Tc_\theta \\
            & \begin{bmatrix}
                M & \theta \\
                \theta^T & 1 
            \end{bmatrix} \succeq 0.
    \end{array}
\end{equation}
In this example, we fix the initial set to $Z = \{z^0\}$ with $z^0 = (1/2, 1/2)$ and consider two different sets for $\Theta$.
We again consider the same $P$ as in the previous example.
Let $\Theta_1 = \left\{ \theta \in \reals^2 \mid \left\| \theta - (1,0) \right\|_2 \le 1/4\right\}$ and $\Theta_2 = \left\{ \theta \in \reals^2 \mid \left\| \theta - (0,1) \right\|_2 \le 1/4\right\}$.
We solve the verification problem \gls{SDP} given by~\eqref{prob:UQP_VP_Thetaonly_SDP} with both $\Theta_1$ and $\Theta_2$, and extract the worst-case parameters, $\theta^\star_1$ and $\theta^\star_2$, and corresponding worst-case iterates, $z^\star_1 = -P^{-1}\theta^\star_1$ and $z^\star_2 = -P^{-1}\theta^\star_2$.
Then, we compare solve the PEP problem with distance from optimality $R^2 = \max\{\lVert (1,0) - z^\star_1\rVert^2, \lVert (0,1) - z^\star_2\rVert^2\}$.
Results are shown in Figure~\ref{fig:uqp_exp2}.
\begin{figure}
    \centering
        \begin{minipage}{\textwidth}
        \centering
\includegraphics[width=0.48\textwidth]{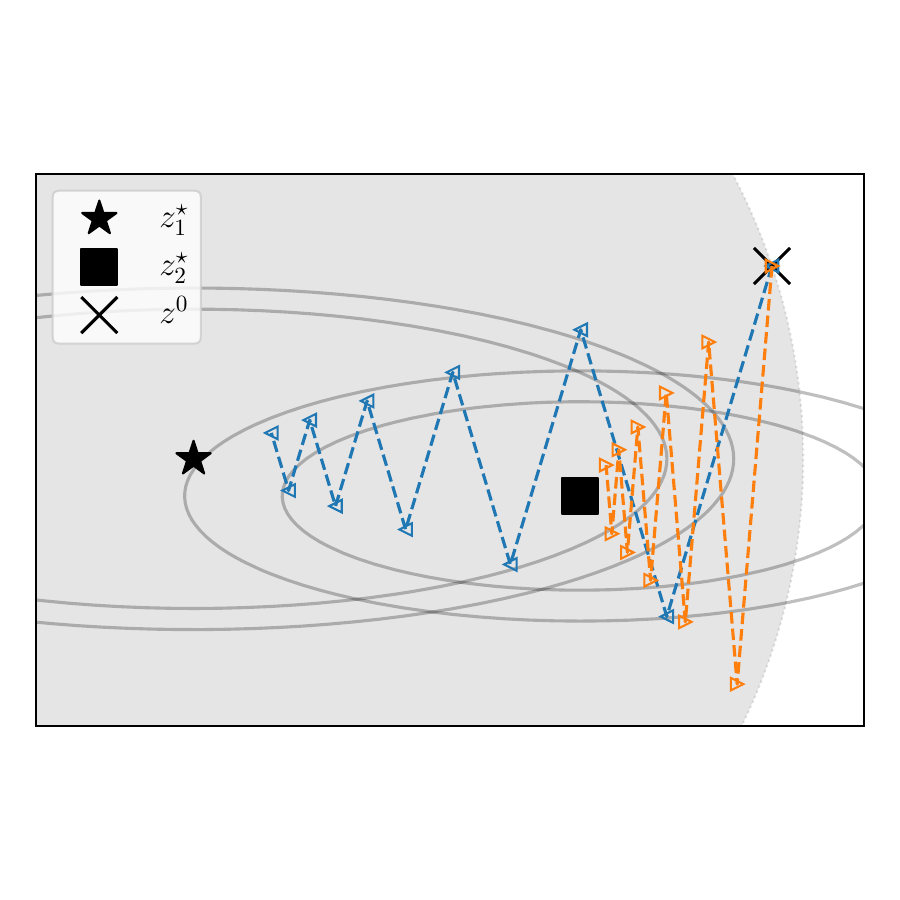}
    \raisebox{1.5em}{\includegraphics[width=0.48\textwidth]{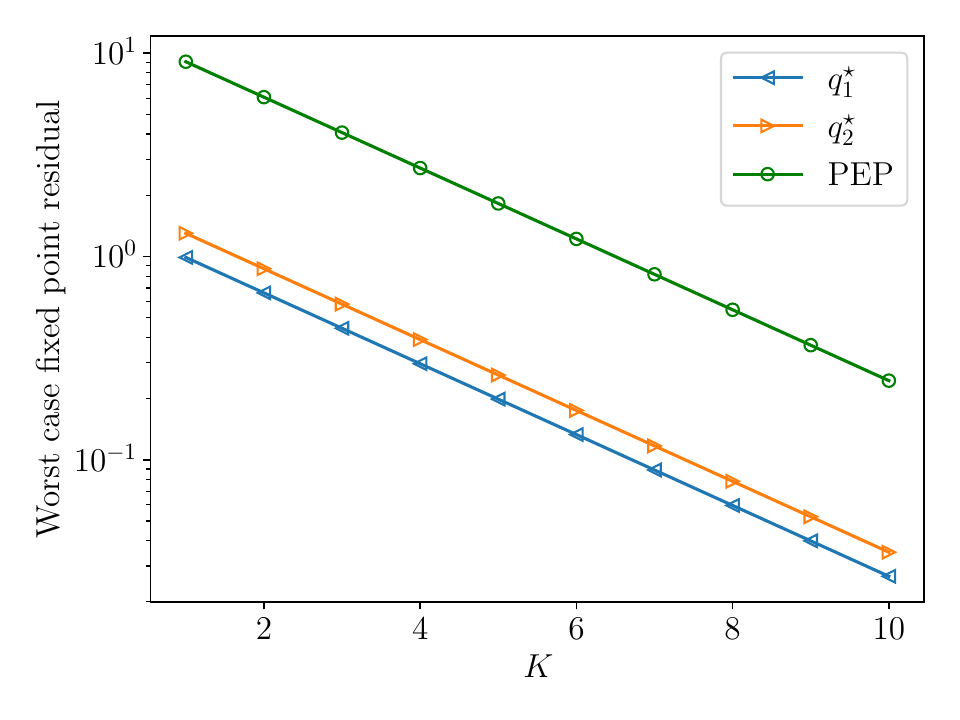}}
    \end{minipage}
    \caption{
Unconstrained QP example with two parameter sets and fixed initial iterate $z^0$. On the left, we show the iterate paths for the worst-case parameter values together with the contour plots of the worst-case functions. The shaded gray area represents the initial condition for the PEP problem, which is large enough to include the initial iterate $z^0$. On the right we plot the worst-case fixed-point residuals for $K=10$ steps. 
}
    \label{fig:uqp_exp2}
\end{figure}

\paragraph{Quadratic term comparison.} 
Lastly, we set $\Theta = \{0\}$ and $Z = \left\{z \in \reals^2 \mid \|z\|\le 1\right\}$, and vary $P$ while maintaining the smoothness and strong convexity parameters. 
To achieve this, we randomly sample orthogonal matrices $Q_1, Q_2$ and define $P_1 = Q_1 P Q_1^T$ and $P_2 = Q_2 P Q_2^T$.
Since we set $\Theta = \{0\}$, the verification problem \gls{SDP} is again of the form~\eqref{prob:UQP_VP_Zonly_SDP}, and $P_1$, $P_2$ enter into the problem through $H$ and $E$.
We solve the verification problem \gls{SDP} with $P_1$ and $P_2$ and show the results in Figure~\ref{fig:uqp_exp3}, in comparison with the PEP framework. 
\begin{figure}
    \begin{minipage}{\textwidth}
        \centering
    \raisebox{-.9em}{\includegraphics[width=0.44\textwidth]{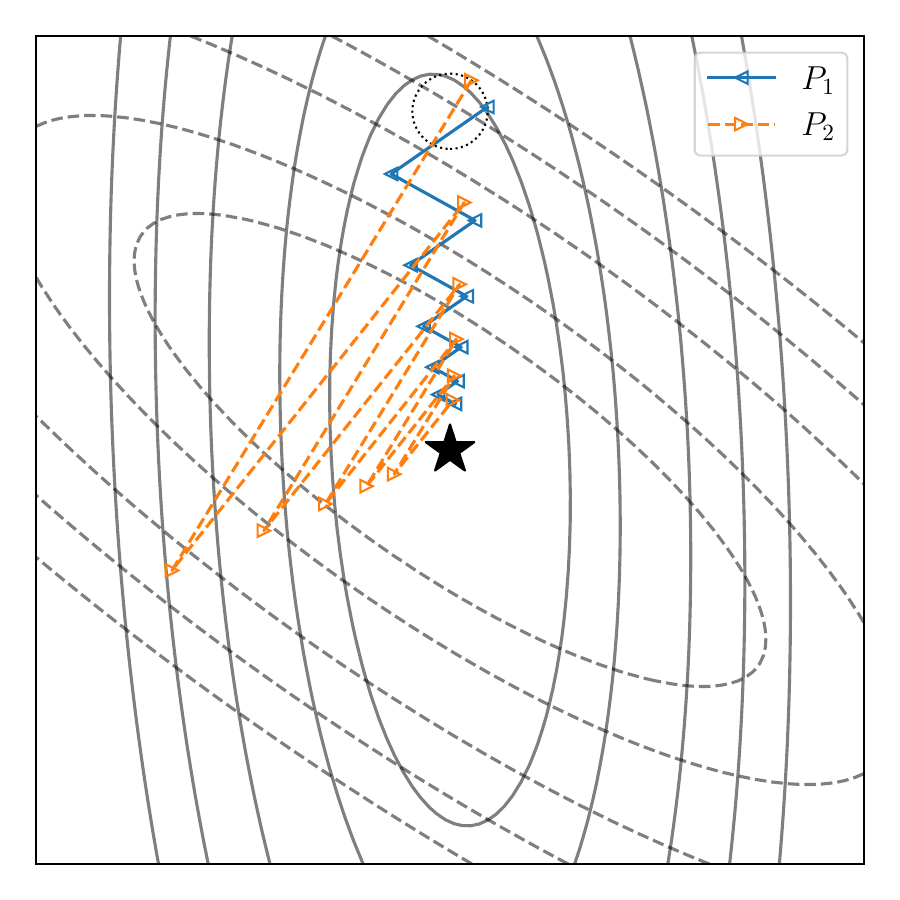}}
    \includegraphics[width=0.48\textwidth]{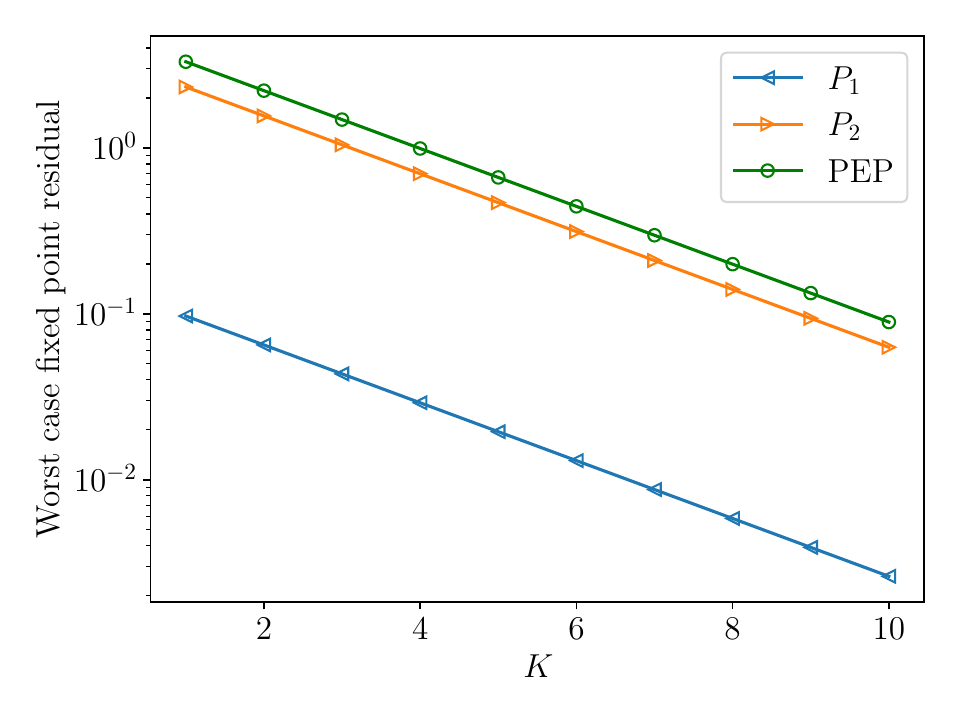}
    \end{minipage}
    \caption{Unconstrained QP example with two different quadratic terms. On the left, we show the worst-case iterate paths for up to $K=10$ steps. On the right, we compare the worst-case residuals to those from the PEP SDP.}
    \label{fig:uqp_exp3}
\end{figure}

\section{Convex relaxation}\label{sec:cvx_relaxation}

\subsection{NP-hardness}
Besides the gradient descent examples for unconstrained QPs presented in Section~\ref{subsec:unconstrainedQP}, where the $S$-lemma provides tight \gls{SDP} relaxations, it is in general \NPhard~to solve problem~\eqref{prob:general_verifyprob}, as shown in Theorem~\ref{thm:nphard}, whose proof appears in Appendix~\ref{apx:nphard}.
\begin{theorem}\label{thm:nphard}
    The performance verification problem~\eqref{prob:general_verifyprob} is \NPhard.
\end{theorem}
In the following sections, we derive strong relaxations for the performance verification problem, by replacing nonconvex components in objective and constraints with semidefinite constraints that are linear in the matrix variables.
With this procedure, we can compute upper bounds to the optimal value of problem~\eqref{prob:general_verifyprob}.
For each set described in this section, we represent the original constraints as set $F$ and the lifted constraints as set $G$.
\reviewChanges{
In each case, $F$ includes exact rank-1 constraints to give the relationships between vector and matrix variables, while $G$ includes the corresponding convex relaxations.
}

\subsection{Objective}\label{sec:sdp_obj}

Proposition~\ref{prop:obj_relax} shows how we replace the quadratic objective $\|z^{K} - z^{K-1}\|_2^2$ with a linear objective in the new matrix variables together with a positive semidefinite constraint to obtain an upper bound. The proof appears in Appendix~\ref{apx:proof_obj_relax}.

    \begin{proposition}[objective reformulation]\label{prop:obj_relax}
        Consider functions $f : \reals^d \times \reals^d \to \reals$ and $g : \symm^{2d \times 2d} \to \reals$ defined as 
        \reviewChanges{
        \begin{equation*}
f(x, y) = \|x - y\|_2^2, \quad
        g(M) = \trace\left(\begin{bmatrix}
            I & -I \\ -I & I
        \end{bmatrix} M \right).
        \end{equation*}
        If the following holds:
        \begin{equation*}
            \begin{bmatrix}
                         \multicolumn{2}{c}{\multirow{2}{*}{$M$}}& x\\
                       \multicolumn{2}{c}{}&y\\
                       x^T & y^T & 1
                    \end{bmatrix} \succeq 0,
        \end{equation*}
        then, $g(M) \ge f(x, y)$.
        }
\end{proposition}

\subsection{Primitive algorithm steps}\label{sec:cvx_algsteps}
Since all fixed-point iteration steps we consider can be written in terms of affine and element-wise maximum steps (Section~\ref{sec:fixedpt_core}), we describe their reformulation in detail.

\paragraph{Affine.} 
For an affine step \eqref{eqn:linearupdate} of the form $Dy = Ax + Bq(\theta)$, we can directly enforce the linear equalities as constraints of the verification problem.
However, as shown in Proposition~\ref{prop:obj_relax}, we introduce new matrix variables in the \gls{SDP} relaxation.
To bound the new matrix variables that correspond to affine steps, we self-multiply the constraints:
\begin{equation}
    (Dy)(Dy)^T = (Ax + Bq(\theta))(Ax + Bq(\theta))^T,
\end{equation}
and this allows us to derive constraints on the matrix variables~\cite{park2017general,jiang2020qcqp_sdp}.
Similarly to equation~\eqref{eqn:matvar_psdrelaxation}, we linearize the vector products by introducing new matrix variables, as shown in Proposition~\ref{prop:affine_relax}, which we prove in Appendix~\ref{apx:proof_affine_relax}.

\begin{proposition}[affine step]\label{prop:affine_relax}
    Consider 
matrices $A \in \reals^{d\times d}$, $B \in \reals^{d \times n}$, $D \in \reals^{d \times d}$, where $D$ is invertible.
    \reviewChanges{
    Let $F, G \subseteq \reals^d \times \reals^d \times \reals^n \times \symm^{d \times d} \times \symm^{(d + n) \times (d + n)}$ be defined by
    \begin{align*}
        F &= \left\{(y,x,q, M_1, M_2) \;\middle|\; Dy = Ax + Bq, \; M_1 = yy^T, \; M_2 = \begin{bmatrix}
            x \\ q
        \end{bmatrix}\begin{bmatrix}
            x \\ q
        \end{bmatrix}^T
        \right\}, \; \text{and}\\
        G &= \left\{(y,x,q, M_1, M_2) \;\middle|\;  
\begin{array}{l}
         Dy = Ax + Bq, \\
         D M_1 D^T = \begin{bmatrix}
             A & B
         \end{bmatrix} M_2 \begin{bmatrix}
             A^T \\ B^T
         \end{bmatrix},\\
         \begin{bmatrix}
             M_1 & y \\
             y^T & 1
         \end{bmatrix} \succeq 0,\;
         \begin{bmatrix}
          \multicolumn{2}{c}{\multirow{2}{*}{$M_2$}}& x\\
        \multicolumn{2}{c}{}&q\\
        x^T & q^T & 1
     \end{bmatrix} \succeq 0
    \end{array}
        \right\}.
    \end{align*}
    Then, $F \subseteq G$.
}
\end{proposition}

\paragraph{Element-wise maximum.}
For element-wise maximum steps~\eqref{eqn:generalmaxconstraints} of the form $y = \max\{x, l\}$, we can relax the constraints using Proposition~\ref{prop:max_relax}, that we prove in Appendix~\ref{apx:proof_max_relax}.

\begin{proposition}[element-wise maximum step]\label{prop:max_relax}
\reviewChanges{
    Let $F, G \subseteq \reals^d \times \reals^d \times \reals^d \times \symm^{3d \times 3d}$ be defined by
    \begin{align*}
        F &= \left\{(y, x, l, M) \;\middle|\; y = \max\{x, l\},~M = \begin{bmatrix}
            y \\ x \\ l
        \end{bmatrix}\begin{bmatrix}
            y \\ x \\ l
        \end{bmatrix}^T \right\}, \quad \text{and}\\
        G &= \left\{ (y, x, l, M) \;\middle|\;
        \begin{array}{l}
     y \geq x, \; y \geq l,\\
     \trace \left( \begin{bmatrix}
        I & -I/2 & -I/2 \\
        -I/2 & 0 & I/2 \\
        -I/2 & I/2 & 0
    \end{bmatrix}M\right) = 0,\\
    \begin{bmatrix}
    \multicolumn{3}{c}{}&y\\
    \multicolumn{3}{c}{\multirow{1}{*}{$M$}}& x\\
    \multicolumn{3}{c}{}&l\\
    y^T & x^T & l^T & 1
\end{bmatrix} \succeq 0.
\end{array}
        \right\}.
    \end{align*}
Then, $F \subseteq G$.
    }
\end{proposition}

All proximal steps considered in this work can be written in terms of these two core steps, as discussed in Section~\ref{sec:fixedpt_extrasteps}.
We summarize the proximal algorithm steps from Section~\ref{sec:fixedpt_extrasteps} and their corresponding relaxations in Table~\ref{tab:proxoper_constraints}.
\begin{table}
  \centering
  \small
  \caption{Proximal algorithm steps for solving~\acrlongpl{PQP}~\cite[\S 6]{parikh2013prox}.}
  \label{tab:proxoper_constraints}
\adjustbox{max width=\textwidth}{
\begin{tabular}{@{}llll@{}}
\toprule
$f$ & $z = \prox_{\lambda f}(v)$ & Constraints & Lifted Constraints\\
\midrule
$x^TAx + b^Tx$ & $z = (A + \lambda I)^{-1}(v - b)$ & $\;\;(A + \lambda I)z = v - b$ &
$\begin{array}{l}
     (A + \lambda I)z = v - b, \\
     (A + \lambda I) M_1 (A + \lambda I)^T = \begin{bmatrix}
         I & -I
     \end{bmatrix} M_2 \begin{bmatrix}
         I \\ -I
     \end{bmatrix},\\
     \begin{bmatrix}
         M_1 & z \\
         z^T & 1
     \end{bmatrix} \succeq 0,\;
     \begin{bmatrix}
      \multicolumn{2}{c}{\multirow{2}{*}{$M_2$}}& v\\
    \multicolumn{2}{c}{}&b\\
    v^T & b^T & 1
 \end{bmatrix} \succeq 0
\end{array}$\\
\midrule
$I_{\reals^n_{+}}$ & $z = (v)_{+}$ & $\begin{array}{l}z \ge 0,\; z \ge v,\\z^T(z-v) = 0\end{array}$ & $\begin{array}{l}z \geq 0,\; z \geq v,\\
\trace\left(\begin{bmatrix}
    I & -I/2 \\ -I/2 & 0
\end{bmatrix}M\right) = 0, \\
 \begin{bmatrix}
      \multicolumn{2}{c}{\multirow{2}{*}{$M$}}& z\\
    \multicolumn{2}{c}{}&v\\
    z^T & v^T & 1
 \end{bmatrix} \succeq 0
\end{array}$\\

\midrule
$I_{[l, u]}$ & $z = \min\left\{u, \max\left\{v, l \right\}\right\}$ & $\begin{array}{l}y \ge l,\; y \ge v,\\(y - l)^T(y-v) = 0\\z \le u,\; z \le y,\\(u - z)^T(y - z) = 0\end{array}$ & 
$\begin{array}{l} y \geq l,\; y \geq v,\; z \leq u,\; z\leq y,\\
\trace\left(\begin{bmatrix}
    I & -I/2 & -I/2 \\ -I/2 & 0 & I/2 \\ -I/2 & I/2 & 0
\end{bmatrix}M_1\right) = 0, \\
\trace\left(\begin{bmatrix}
    I & -I/2 & -I/2 \\ -I/2 & 0 & I/2 \\ -I/2 & I/2 & 0
\end{bmatrix}M_2\right) = 0,\\
\begin{bmatrix}
    \multicolumn{3}{c}{}&y\\
    \multicolumn{3}{c}{\multirow{1}{*}{$M_1$}}& v\\
    \multicolumn{3}{c}{}&l\\
    y^T & v^T & l^T & 1
\end{bmatrix} \succeq 0,
\begin{bmatrix}
    \multicolumn{3}{c}{}&z\\
    \multicolumn{3}{c}{\multirow{1}{*}{$M_2$}}& y\\
    \multicolumn{3}{c}{}&u\\
    z^T & y^T & u^T & 1
\end{bmatrix} \succeq 0
\end{array}$
\\

\midrule

$\|x\|_1$ & $z = \max\{v, \lambda\} - \max\{-v, \lambda\}$ & 
$\begin{array}{l}y \ge \lambda,\; y \ge v\\
    (y - \lambda)^T(y - v) = 0\\
    w \ge \lambda,\; w \ge -v\\
    (w - \lambda)^T(w + v) = 0\\
    z = y - w\end{array}$ & 
    
    $\begin{array}{l} y \geq \lambda,\; y \geq v,\; w \geq \lambda,\; w \geq -v,\;z = y - w,\\
    \trace\left(\begin{bmatrix}
        I & -I/2 & -(1/2)\ones \\ -I/2 & 0 & (1/2)\ones \\ -(1/2)\ones^T & (1/2)\ones^T & 0
    \end{bmatrix}M_1\right) = 0, \\
    \trace\left(\begin{bmatrix}
         I & I/2 & -(1/2)\ones \\ I/2 & 0 & -(1/2)\ones \\ -(1/2)\ones^T & -(1/2)\ones^T & 0
    \end{bmatrix}M_2\right) = 0,\\
    M_3 = \begin{bmatrix}
        I & -I
    \end{bmatrix} M_4 \begin{bmatrix}
        I \\ 
        -I
    \end{bmatrix},\\
    \begin{bmatrix}
        \multicolumn{3}{c}{}&y\\
        \multicolumn{3}{c}{\multirow{1}{*}{$M_1$}}& v\\
        \multicolumn{3}{c}{}&\lambda\\
        y^T & v^T & \lambda & 1
    \end{bmatrix} \succeq 0,
    \begin{bmatrix}
        \multicolumn{3}{c}{}&w\\
        \multicolumn{3}{c}{\multirow{1}{*}{$M_2$}}& v\\
        \multicolumn{3}{c}{}&\lambda\\
        w^T & v^T & \lambda & 1
    \end{bmatrix} \succeq 0\\
    \begin{bmatrix}
             M_3 & z \\
             z^T & 1
         \end{bmatrix} \succeq 0,\; \begin{bmatrix}
        \multicolumn{3}{c}{}&z\\
        \multicolumn{3}{c}{\multirow{1}{*}{$M_4$}}& y\\
        \multicolumn{3}{c}{}&w\\
        z^T & y^T & w^T & 1
    \end{bmatrix} \succeq 0
    \end{array}$
    \\

\bottomrule
\end{tabular}}
\end{table}

\subsection{Initial iterate and parameter sets}
While the sets in Section~\ref{sec:verifyprob_initsets} are convex, we now express them in terms of the lifted \gls{SDP} variables by introducing additional constraints to obtain tight \gls{SDP} relaxations for problem~\eqref{prob:general_verifyprob}.

\paragraph{Hypercubes.}
For hypercubes, we derive constraints using Proposition~\ref{prop:hypercube_relax}, whose proof appears in Appendix~\ref{apx:proof_hypercube_relax}.

\begin{proposition}[hypercubes]\label{prop:hypercube_relax}
    Consider vectors~$l, u \in \reals^d$.
\reviewChanges{
    Let $F, G \subseteq \reals^d \times \symm^{d \times d}$ be defined by
    \begin{align*}
        F &= \{ (z, M) \mid l \leq z \leq u, \; M = zz^T \},\quad \text{and}\\
G &= \left\{ (z, M) \;\middle\vert\; 
\begin{array}{ll}
            & uz^T - ul^T - M + zl^T \ge 0,\\
            & u u^T - u z^T - z u^T + M \ge 0,\\
            & M - z l^T - lz^T + ll^T \ge 0,
            \end{array}
        \quad \begin{bmatrix}
            M & z \\
            z^T & 1
        \end{bmatrix} \succeq 0
        \right\},
    \end{align*}
    \reviewChanges{
    where the matrix inequalities $\ge$ are taken elementwise.}
    Then, $F \subseteq G$.
    }
\end{proposition}

\paragraph{Polyhedra.}
For polyhedral constraints, we derive constraints using Proposition~\ref{prop:polyhedra_relax}, which we prove in Appendix~\ref{apx:proof_polyhedra_relax}.
\begin{proposition}[polyhedra]\label{prop:polyhedra_relax}
    Consider $A \in \reals^{m \times d}$, $b \in \reals^m$.
\reviewChanges{
    Let $F, G \subseteq \reals^d \times \reals^m \times \symm^{(d + m) \times (d + m)}$ be defined by
    \begin{align*}
        F &= \left\{(z, s, M) \mid Az + s = b, \; s \ge 0, \; M = \begin{bmatrix} z \\ s \end{bmatrix} \begin{bmatrix} z \\ s \end{bmatrix}^T \right\},\quad \text{and}\\
        G &= \left\{ 
            (z, s, M) \;\middle\vert\; 
        \begin{array}{l}
             Az + s = b, \; s \geq 0  \\
             \begin{bmatrix}
                A & I
            \end{bmatrix} M \begin{bmatrix}
                A^T \\ I
            \end{bmatrix} = bb^T, \quad
            \begin{bmatrix}
      \multicolumn{2}{c}{\multirow{2}{*}{$M$}}& z\\
            \multicolumn{2}{c}{}&s\\
            z^T & s^T & 1
         \end{bmatrix} \succeq 0, 
\end{array}
        \right\}.
    \end{align*}
Then, $F \subseteq G$.
    }
\end{proposition}
\paragraph{$\ell_p$-ball.} We handle the cases where $p = 1,2,\infty$ separately as each one requires different types of constraints. 
For $p=1$, we introduce an auxiliary variable $u \in \reals^d$ to encode~${\|z^0 - c\|_1 \leq r}$ with the following linear constraints,
\begin{equation*}
    z^0 - c \leq u, \quad -z^0 + c \leq u, \quad \ones^T u \leq r.
\end{equation*}
By considering the stacked variable~$(z^0,u)$, we obtain a polyhedral constraint with
\begin{equation*}
    A = \begin{bmatrix}
        I & -I \\
        -I & -I \\
        0 & \ones^T
    \end{bmatrix}, \quad
    b = \begin{bmatrix}
        c \\
        -c \\
        r
    \end{bmatrix}.
\end{equation*}
which we can reformulate with Proposition~\ref{prop:polyhedra_relax}.
The case of $p=\infty$ corresponds to a hypercube with $l = c - r\ones$ and $u = c + r\ones$, which we can reformulate with Proposition~\ref{prop:hypercube_relax}.
For $p = 2$, we apply Proposition~\ref{prop:l2ball_relax}, which is proven in Appendix~\ref{apx:proof_l2_ball_relax}.
\begin{proposition}[$\ell_2$-ball]\label{prop:l2ball_relax}
    Consider $c \in \reals^d$ and $r \in \reals$.
\reviewChanges{
    Let $F, G \subseteq \reals^d \times \symm^{2d \times 2d}$ be defined by
    \begin{align*}
        F &= \left\{(z, M) \mid \| z - c\| \leq r, \quad M = zz^T \right\},\quad \text{and} \\
        G &= \left\{ (z, M) \;\middle\vert\;
        \begin{array}{l}
     \trace(M) - 2c^T z + c^T c \leq r^2, \quad 
\begin{bmatrix}
        M & z \\
        z^T & 1
    \end{bmatrix} \succeq 0
\end{array}
        \right\}.
    \end{align*}
Then, $F \subseteq G$.
    }
\end{proposition}

\subsection{Tightening the relaxation}\label{subsec:tightening_sdp}
\reviewChanges{To tighten the relaxation, we use the \gls{RLT}~\cite{sherali1995reformulationconvexification,sherali2010rlt,anstreicher2009sdprlt} to introduce additional inequalites.}
The main idea is to combine known bounds on the any problem variable to form valid inequalities involving the new matrix variables.
Consider a variable $z \in \reals^d$ and its upper bound $\overline{z}$ and lower bound $\underline{z}$. 
Analogously to Proposition~\ref{prop:hypercube_relax}, we can rearrange the bound constraints as follows
\begin{equation*}
    \underline{z} \le z \le \overline{z} \quad \iff \quad 
    \begin{array}{ll}
        &(\overline{z} - z)(\overline{z} - z)^T \ge 0\\
        &(z - \underline{z})(z - \underline{z})^T \ge 0\\
        &(\overline{z} - z)(z - \underline{z})^T \ge 0
    \end{array}
\end{equation*}
where we cross-multiplied the lower and upper bound inequalities.
By replacing products $zz^T$ with the symmetric positive semidefinite matrix $M \in \symm^{d \times d}_{+}$, we obtain constraints
\begin{equation}
    \begin{array}{ll}\label{eq:rltderivation}
    &\overline{z} \overline{z}^T - \overline{z} z^T - z \overline{z}^T + M \ge 0\\
    & M - z \underline{z}^T - \underline{z}z^T + \underline{z}\underline{z}^T \ge 0\\
    & \overline{z}z^T - \overline{z}\underline{z}^T - M + z \underline{z}^T \ge 0\\
    & M \succeq 0,
    \end{array}
\end{equation}
which we directly include in our relaxation.
\reviewChanges{
Note that the original constraints $\underline{z} \le z \le \overline{z}$ are implied by the set of constraints~\eqref{eq:rltderivation} \cite[Proposition 1]{sherali1995reformulationconvexification}.
}
By knowing the upper and lower bounds on the problem variables, this procedure allows us to tighten the relaxation in terms of the lifted variables.

\paragraph{Bound propagation.}
The \gls{RLT} requires lower and upper bounds on every variable in order to construct tightening inequalities.
First, we use the initial sets $Z, \Theta$ to derive bounds on the initial iterate $z^0$ and parameter $\theta$.
When the sets are either hypercubes or $\ell_p$-ball sets, the initial bounds can be efficiently computed.
In the hypercube case, $Z = \{z \mid l \leq z \leq u\}$, the initial bounds are given by $\underline{z} = l,~\overline{z} = u$.
In the $\ell_p$-ball case, $Z = \{z \mid \left\| z - c \right\|_p \leq r \}$, the initial bounds are given by
$\underline{z} = c - r,~\overline{z} = c + r$.

We need to efficiently propagate these initial bounds across iterate steps.
\begin{itemize}
\item {\bf Affine steps.} First consider an affine step $Dy = Ax + B q(\theta)$ with known lower bounds $\underline{x}, \underline{q}(\theta)$ and upper bounds $\overline{x}, \overline{q}(\theta)$.
To propagate these bounds on $y$, we use the invertibility assumption of $D$ and write
\begin{equation}\label{eq:direct_affine_step}
    y = \tilde{A} x + \tilde{B} q(\theta),
\end{equation}
where we compute $\tilde{A} = D^{-1}A$ and $\tilde{B} = D^{-1}B$ using a factorization of matrix $D$.
We now bound the terms in equation~\eqref{eq:direct_affine_step} individually and then sum the bounds together as commonly done in {\it interval arithmetic}~\cite{gowal2018boundprop}.
Given bounds on $x$ denoted by $\underline{x}, \overline{x}$, the bounds for $\tilde{A}x$ are computed via \cite[Section 3]{gowal2018boundprop}:
\begin{equation}\label{eq:standard_affine_boundprop}
\begin{array}{ll}
\tilde{A}x &\le (1/2)\left(\tilde{A} (\overline{x} + \underline{x}) + |\tilde{A}| (\overline{x} - \underline{x}) \right)\\
    \tilde{A}x &\ge (1/2)\left(\tilde{A} (\overline{x} + \underline{x})- |\tilde{A}| (\overline{x} - \underline{x}) \right),
\end{array}
\end{equation}
where $|\tilde{A}|$ is the element-wise absolute value of $\tilde{A}$.

If $x$ lives in an $\ell_p$-ball of radius $r$ centered at $c$, then we can construct tighter bounds on each component of $\tilde{A}x$ as follows \cite[Section 3.2]{xu2020automatic}:
\begin{equation}\label{eq:lp_affine_boundprop}
\begin{array}{ll}
    (\tilde{A}x)_i \le r \left\| \tilde{A}_i^T \right\|_q + \tilde{A}_i^T c\\
    (\tilde{A}x)_i \ge -r \left\| \tilde{A}_i^T \right\|_q + \tilde{A}_i^T c,
\end{array}
\end{equation}
for $i=1,\dots,d$. Here, where $\tilde{A}_i^T$ is the $i$-th row of $\tilde{A}$ and $q$ is the dual norm of $p$ satisfying $1/p + 1/q = 1$.
We can use~\eqref{eq:standard_affine_boundprop} or $\eqref{eq:lp_affine_boundprop}$ to bound $\tilde{B} q(\theta)$ in a similar way.

\item {\bf Element-wise maximum steps.} 
Consider an element-wise maximum step given by $y = \max\{x, l(\theta)\}$.
With known lower bounds $\underline{x}, \underline{l(\theta)}$ and upper bounds $\overline{x}, \overline{l(\theta)}$, we can compute bounds on $y$.
Since the $\max$ function is element-wise nondecreasing, the bounds propagate as \cite[Section 3]{gowal2018boundprop}:
\begin{equation}\label{eq:max_boundprop}
    \underline{y} = \max\{\underline{x}, \underline{l(\theta)}\}, \quad \overline{y} = \max\{\overline{x}, \overline{l(\theta)}\}.
\end{equation}
\end{itemize}

In the verification problem, we use the techniques described by~\eqref{eq:standard_affine_boundprop},~\eqref{eq:lp_affine_boundprop}, and~\eqref{eq:max_boundprop} to derive lower and upper bounds on every variable across all iterations $k$.
We use these bounds to form~\gls{RLT} inequalities and tighten the relaxation.

\paragraph{Triangle relaxation.}
In addition to the variable bounds with \gls{RLT}, we also compute further bounds for the element-wise maximum steps.
Consider an element-wise maximum step $y = \max\{x, l(\theta)\}$, $y, x, l(\theta) \in \reals^d$, with precomputed lower bounds $\underline{y}, \underline{x}$ and upper bounds $\overline{y}, \overline{x}$.
Note that the bounds on $l(\theta)$ enter through the precomputed bounds on $y$ as shown in~\eqref{eq:max_boundprop}.
We can introduce another valid inequality with the so-called \emph{triangle} relaxation~\cite{ehlers2017planet}:
\begin{equation}\label{eqn:planet_ineq}
    y \leq \frac{\overline{y} - \underline{y}}{\overline{x} - \underline{x}}(x-\underline{x}) + \underline{y},
\end{equation}
For simplicity, we rewrite~\eqref{eqn:planet_ineq} as
\begin{equation}\label{eqn:planet_vec}
    y \leq Ex + c,
\end{equation}
for diagonal matrix $E$ and vector $c$ defined as
\begin{equation*}
    E_{ii} = \frac{\overline{y}_i - \underline{y}_i}{\overline{x}_i - \underline{x}_i}, \quad c_i = -\frac{\overline{y}_i - \underline{y}_i}{\overline{x}_i - \underline{x}_i}\underline{x}_i + \underline{y}_i,
\end{equation*}
for $i=1,\dots,d$.
To incorporate the triangle relaxation into our relaxed problem, we use a similar technique to the~\gls{RLT} inequalities.
That is, we multiply each inequality in~\eqref{eqn:planet_vec} by $\overline{x} - x \geq 0$ to obtain bounds on the matrix variables.
Explicitly,
\begin{equation}\label{eq:planet_rlt_exact}
    (Ex + c - y)(\overline{x} - x)^T \geq 0 \implies Ex \overline{x}^T - Exx^T + c\overline{x}^T - cx^T -y\overline{x}^T + yx^T \geq 0.
\end{equation}
In Proposition~\ref{prop:max_relax}, we introduce a lifted matrix variable $M$ for the element-wise maximum steps.
We can rewrite the inequalities from~\eqref{eq:planet_rlt_exact} in terms of this same $M$.
The nonnegative expression in~\eqref{eq:planet_rlt_exact} is a $d \times d$ matrix, so we show the constraint corresponding to the $ij$-th element.
In terms of $M$ from Proposition~\ref{prop:max_relax}, the inequalities can be written as
\begin{equation}\label{eq:planet_rlt_lifted}
    (\overline{x}_j E_i)^T x + c_i \overline{x}_j  - c_i x_j - \overline{x}_j y_i + \trace\left(\begin{bmatrix}
        0 & e_i e_j^T & 0 \\
        0 & -E_{ii} e_i e_j^T & 0 \\
        0 & 0 & 0
    \end{bmatrix} M \right) \geq 0, \quad i,j = 1,\dots, d,
\end{equation}
where $e_i \in \reals^d$ equals 1 at index $i$ and 0 elsewhere.
\reviewChanges{
The inequalities in~\eqref{eq:planet_rlt_lifted} also show that we can reuse the lifted matrix variables~$M$ across different triangle relaxations. 
We discuss this idea of reusing matrix variables further with a complete~\gls{SDP} formulation in Appendix~\ref{apx:sdp_coupling}.
}

\section{Numerical examples}\label{sec:experiments}
For every experiment, we compare three main metrics:
\begin{itemize}
    \item \textbf{Verification problem SDP objective} ($\text{VPSDP}$).
    We solve the SDP relaxation of the verification problem, denoted VPSDP, to compute the upper bound $\text{VPSDP}$ on the worst-case fixed-point residual over the variations of initial iterates and problem parameters.
We apply the \reviewChanges{\gls{RLT}, triangle relaxation, and reuse matrix variables as} discussed in Section~\ref{subsec:tightening_sdp} and Appendix~\ref{apx:sdp_coupling}.
    \item \textbf{PEP objective} ($\text{PEP}$).
    We compute the worst-case objective from the PEP framework, denoted as $\text{PEP}$. To apply the PEP framework, we need to provide an upper bound $R$ on the initial distance to optimality, \ie, $\lVert z^0 - z^\star\rVert^2 \le R$. However, as discussed in Section~\ref{sec:PEPdiff}, for every $\theta$ we have a different $z^\star(\theta)$ and it is not obvious which upper bound to provide.    
To estimate $R$, we consider $N=10,000$ problem instances and initial iterates by sampling $\{(z^0_i,\theta_i)\}_{i=1}^N$ uniformly from the set $Z \times \Theta$.
    For each instance, we solve the \gls{PQP} to optimality to find $z^\star(\theta_i)$ and compute $R = \max_{i=1,\dots,N}\lVert z^0_i - z^\star(\theta_i)\rVert$.
Similarly to our approach in Section~\ref{sec:PEPdiff}, we solve the PEP SDPs by encoding the different algorithms using the PEPit toolbox \cite{goujad2022pepit}.
    Specifically, we compute smoothness and strong convexity parameters for the class of quadratic functions~\cite{bousselmi2022pepqp,bousselmi2023pep_linop}.
\item \textbf{Sample maximum objective.} ($\text{SM}$)
    For each of the $N$ sampled problems, we also run the first-order method and compute the fixed-point residual for $k=1,\dots, K$ steps.
    We then compute the sample maximum fixed-point residual for each $k$ and report this as the sample maximum worst-case objective $\text{SM}$.
    This value serves as a lower bound to both the VPSDP and PEP objectives.
\end{itemize}
We also report:
\begin{itemize}
    \item \textbf{Computation times for solving VPSDP.}
    We report the solve time in seconds for each VPSDP.
    \item \textbf{Improvement over PEP.}
We measure the improvement of over PEP, we report the ratio between its objective $\text{PEP}$ and the objective of our relaxation $\text{VPSDP}$ and the lower bound obtained from the sample maximum $\text{SM}$:
    \begin{equation}
        \frac{\text{PEP}}{\text{VPSDP}} \quad \text{and}\quad \frac{\text{PEP}}{\text{SM}}.
    \end{equation}
\end{itemize}
We present full numerical results for every experiment in Appendix~\ref{apx:fullresulttables}.

\paragraph{Computational scalability.}
We use one of our experiments to show the difficulty in solving the nonconvex verification problem exactly.
In Section~\ref{sec:lasso_exp}, we set up and solve the nonconvex verification problem as a \gls{QCQP} with a fixed time limit.
In doing so, we apply bound propagation techniques to construct lower and upper bounds on each variable.
After the time limit, we report the best lower and upper bound on the optimal objective.

\paragraph{Software.}
All examples are written in Python 3.10. 
To solve the \glspl{SDP}, we use MOSEK version 10.1, specifically through the MOSEK Fusion API for Python~\cite{mosek}.
We set the primal feasibility, dual feasibility, and duality gap tolerances all as $10^{-7}$.
To compare our verification problem with PEP, we use the PEPit toolbox~\cite{goujad2022pepit}, which interfaces with MOSEK through CVXPY~\cite{diamond2016cvxpy,agrawal2018rewriting}.
To solve the nonconvex verification problem exactly in Section~\ref{sec:lasso_exp}, we use Gurobi version 10.0 \cite{gurobi}.
All computations were run on the Princeton HPC Della Cluster with 16 CPU cores.
The code for all experiments can be found at 
\begin{center}
\url{https://github.com/stellatogrp/sdp_algo_verify}.
\end{center}

\subsection{Nonnegative least squares}\label{subsec:nnls}
For a matrix $A \in \reals^{m\times n}$ with $m \geq n$ and a parameterized vector $b(\theta) \in \reals^m$, the \gls{NNLS} problem is defined as:
\begin{equation}\label{prob:NNLS}
	\begin{array}{ll}
		\text{minimize} & (1/2)\lVert Ax - b(\theta) \rVert^2 \\
		\text{subject to} & x \geq 0.
	\end{array}
\end{equation}
We analyze projected gradient descent with step size $t > 0$, presented in Algorithm~\ref{alg:proj_gd}.
\begin{algorithm}
  \caption{Projected gradient descent for the \gls{NNLS} problem~\eqref{prob:NNLS}.} 
	\label{alg:proj_gd}
  \begin{algorithmic}[1]
      \State {\bf Given} step size $t > 0$, problem data $(A, b(\theta))$, $K$ number of iterations, initial iterate $z^0$
      \For{$k=0, \dots, K-1$}
        \State $y^{k+1} = (I-tA^TA)z^k + tA^T b(\theta)$
        \State $z^{k+1} = (y^{k+1})_+$
      \EndFor
      \State \textbf{return} $z^K$
  \end{algorithmic}
\end{algorithm}
To choose the step size, we consider the minimum and maximum eigenvalues of $A^TA$, denoted $\mu$ and $L$, respectively.
It is well-known that, for $t < 2/L$, the algorithm converges as the gradient step is a contractive operator~\cite[Chapter 9]{boyd2004convex}\cite[Chapter 2]{lecturesnesterov} and the projection is nonexpansive~\cite[Section 5.1]{ryu2016monotoneprimer}~\cite[Section 2.3]{parikh2013prox}.
Furthermore, 
when $\mu > 0$, choosing $t^\star = 2/(\mu + L)$ is optimal in terms of minimizing the worst-case convergence rate 
(Proposition~\ref{prop:optimal_t}).

We consider two instances where the objective is nonstrongly convex and strongly convex.
First, we consider fixed step sizes across $K$ iterations.
Specifically, we grid 5 different step sizes in the range $[1/L, 2/L]$, and for the strongly convex case, we additionally include $t^\star = 2 / (\mu + L)$.
Second, we consider the silver step size schedule, a method that aims to accelerate gradient descent via carefully selected, non-constant, step sizes~\cite{altschuler2023acceleration,altschuler2023acceleration2}.
The silver step size schedule for strongly convex problems appears in~\cite[Section 3]{altschuler2023acceleration}, and for nonstrongly convex problems appears in~\cite[Section 2]{altschuler2023acceleration2}.
\reviewChanges{
\paragraph{Full \gls{SDP} example.} To demonstrate an example overall construction, we show the full VPSDP for the \gls{NNLS} example in Appendix~\ref{apx:sdp_coupling} for $K=2$ iterations.
Through this example, we show the composition of multiple steps by deriving a complete~\gls{SDP} formulation.
}
\paragraph{Problem setup.}
We randomly generate $A \in \reals^{60 \times 40}$ such that $\mu=0$, $L=100$ for the nonstrongly convex case and $\mu=20$, $L=100$ for the strongly convex case.
We choose $b(\theta)$ to be an $\ell_2$-ball of radius $0.5$ centered at $c$, where $c_i = 30, ~i = 1,\dots, 30$ and 0 otherwise.
For each candidate step size schedule, we run the verification problem with initial iterate set~$Z = \{0\}$.
We run the fixed step size experiments up to $K=10$, and we run the silver step size experiments up to $K=7$ as the recursive definition reaches the best performance when $K$ is one less than a power of two.

\begin{figure}[h]
\centering
    \includegraphics[width=0.8\textwidth]{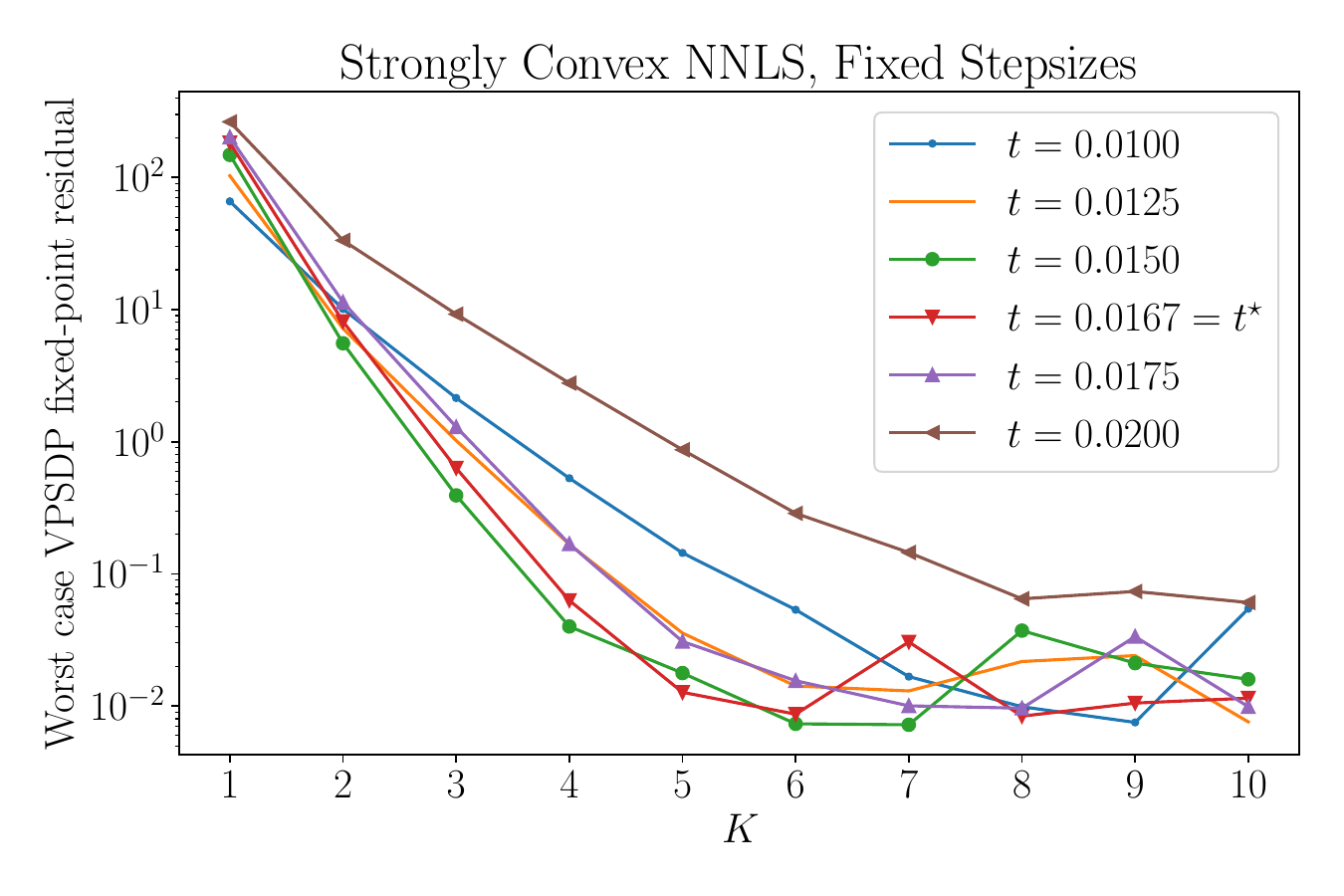}
    \caption{Results for the strongly convex NNLS instance with $\mu=20, L=100$, for fixed step sizes. Full results are in Tables~\ref{tab:nnls_fixed_pt1} and~\ref{tab:nnls_fixed_pt2}.}
    \label{fig:strong_NNLS_gridt}
\end{figure}

\begin{figure}
\centering
    \includegraphics[width=0.8\textwidth]{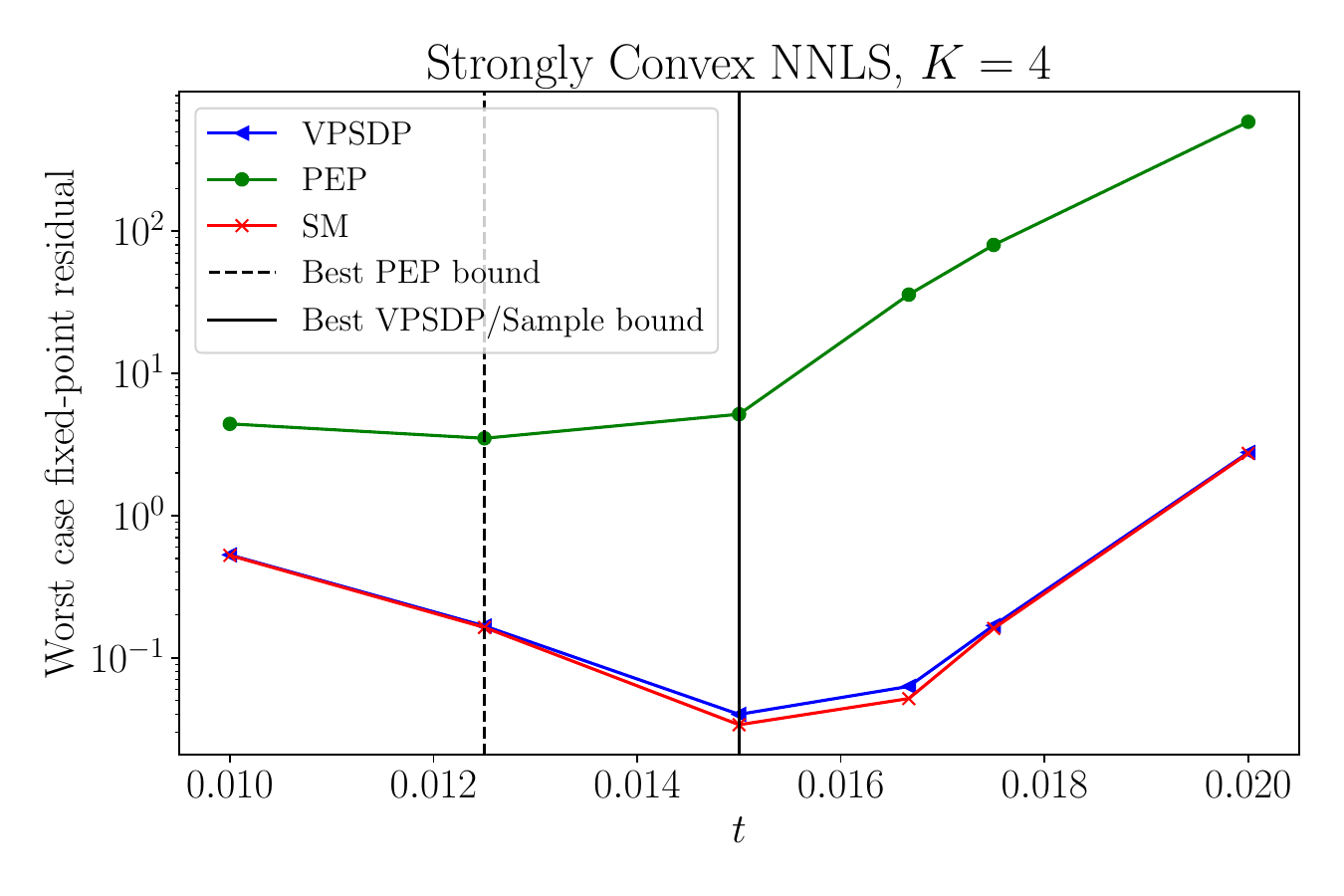}
    \caption{Results to show the effect of $t$ at $K=4$ for the same strongly convex instance as in Figure~\ref{fig:strong_NNLS_gridt}. There is a discrepancy between the PEP and VPSDP bounds on the $t$ that leads to the smallest residual.}
    \label{fig:NNLS_fixed_K4}
\end{figure}

\begin{figure}
\centering
    \includegraphics[width=0.8\textwidth]{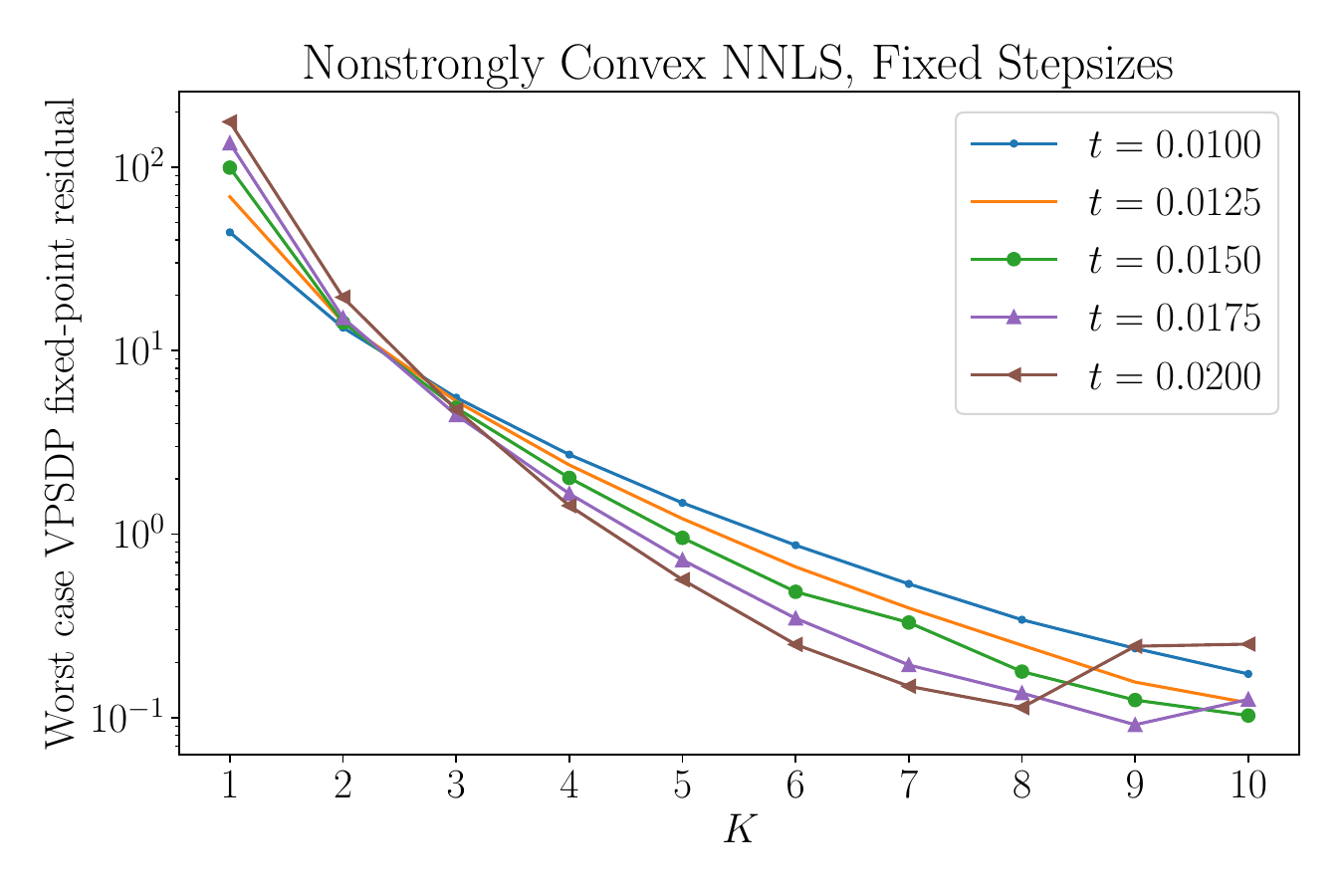}
    \caption{Results for the nonstrongly convex NNLS instance $L=100$ for fixed step sizes. Full results are in Tables~\ref{tab:nnls_nonstrong_fixed_pt1} and~\ref{tab:nnls_nonstrong_fixed_pt2}.}
    \label{fig:nonstrong_NNLS_gridt}
\end{figure}

\begin{figure}
\centering
    \includegraphics[width=0.7\textwidth]{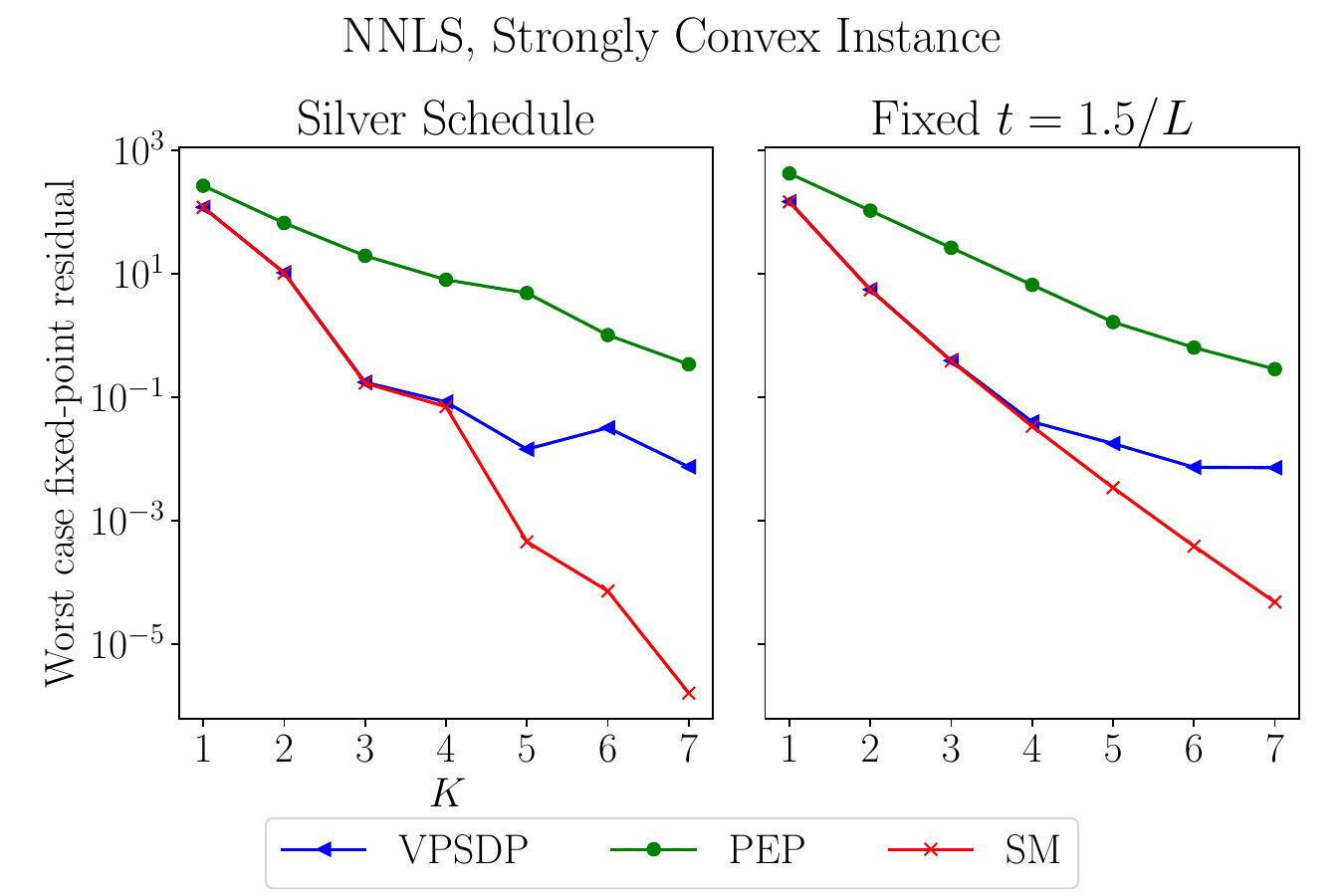}
    \caption{Results for the silver schedule on the strongly convex NNLS instance compared to a single fixed step size schedule. Silver results are presented in Table~\ref{tab:nnls_silver_strongcvx}. The VPSDP bound for the silver schedule is lower at $K=7$ but the relaxation gets looser as the sample maximum approaches 0.}
    \label{fig:NNLS_silver_strongcvx}
\end{figure}

\begin{figure}
\centering
    \includegraphics[width=0.7\textwidth]{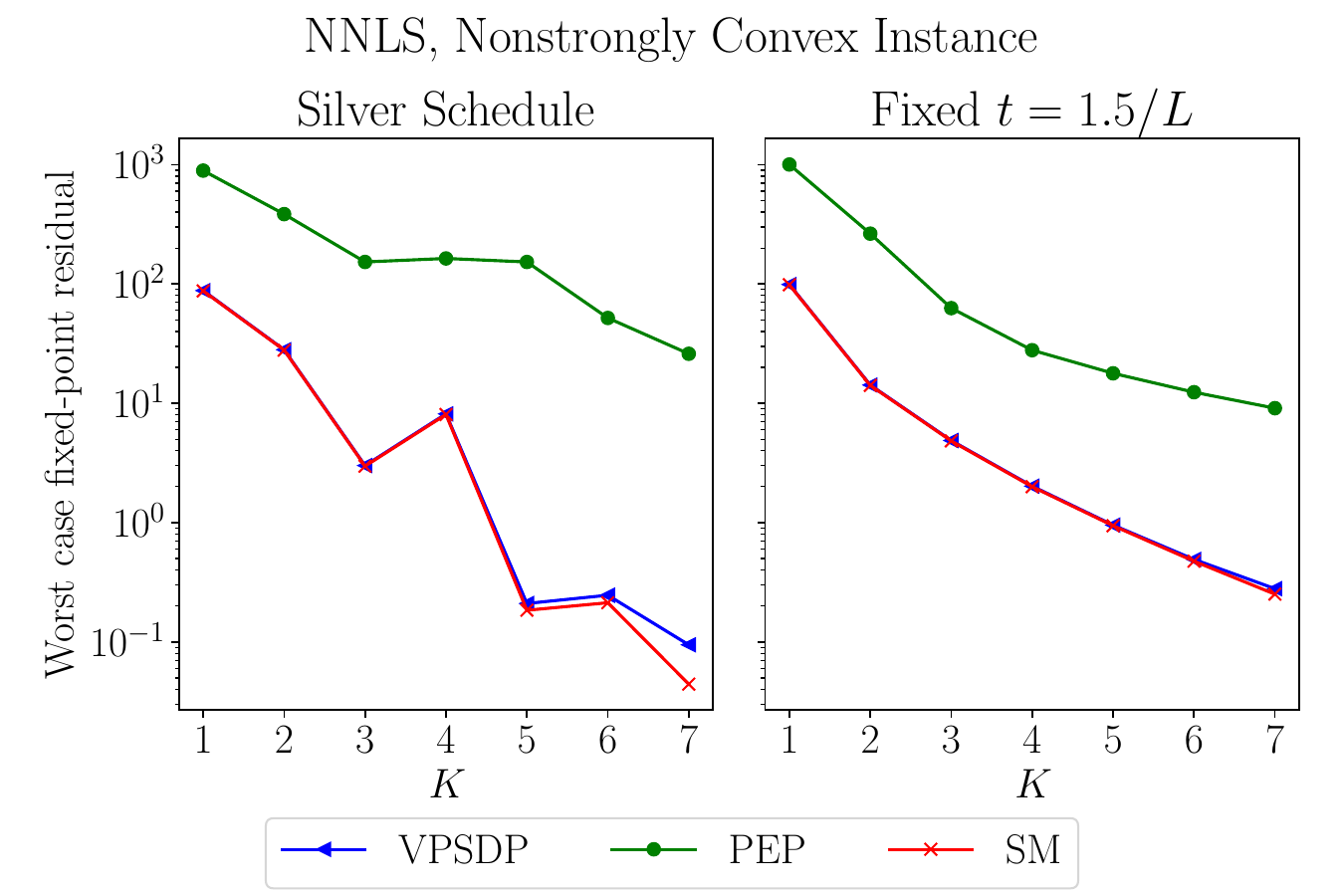}
    \caption{Results for the silver schedule on the nonstrongly convex NNLS instance compared to a single fixed step size schedule. Silver results are presented in Table~\ref{tab:nnls_silver_nonstrongcvx}.
    The VPSDP bound for the silver schedule at $K=7$ is lower than the sample maximum bound for the fixed step size case.
    The overall convergence is slower when compared to the strongly convex case so the relaxation gap is much smaller.}
    \label{fig:NNLS_silver_nonstrongcvx}
\end{figure}

\paragraph{Fixed step size schedule results.}
The nonstrongly convex results are presented in Figure~\ref{fig:nonstrong_NNLS_gridt} and Tables~\ref{tab:nnls_nonstrong_fixed_pt1} and~\ref{tab:nnls_nonstrong_fixed_pt2}, while the strongly convex results are presented in Figure~\ref{fig:strong_NNLS_gridt} and Tables~\ref{tab:nnls_fixed_pt1} and~\ref{tab:nnls_fixed_pt2}.
Across the step sizes, our VPSDP shows a significant improvement in the estimate of the worst-case objective over PEP.
In particular, in both cases, with a step size $t=2/L$, the PEP bound remains constant for all $K$.
This implies that there exists an $L$-smooth, $\mu$-strongly convex function where proximal gradient descent does not converge, which agrees with the classical analysis~\cite[Chapter 9]{boyd2004convex}, \cite[Chapter 2]{lecturesnesterov}.
However, our VPSDP captures that proximal gradient descent still converges for these particular \glspl{PQP}.

In Figure~\ref{fig:NNLS_fixed_K4}, we look at the strongly convex instance at $K=4$ across the different step sizes.
There is a discrepancy between the step size that gives the lowest PEP bound and the lowest VPSDP bound.
When paired with the lower bound from the sample maximum, SM, the VPSDP is able to determine the correct step size choice when compared to PEP.
The relative scales of the two instances show faster convergence in the strongly convex instance, as expected.
When comparing to the sample maximum bound across the step sizes, the VPSDP relaxation gets looser as $K$ grows.
In particular, as the sample maximum fixed point residual gets close to 0, the corresponding VPSDP bound becomes looser.

\paragraph{Silver step size schedule results.}
We present the results with the silver step size schedule for a strongly convex problem in Figure~\ref{fig:NNLS_silver_strongcvx} and Table~\ref{tab:nnls_silver_strongcvx}, and for a nonstrongly convex problem in Figure~\ref{fig:NNLS_silver_nonstrongcvx} and Table~\ref{tab:nnls_silver_nonstrongcvx}.
In both cases the VPSDP gives a tighter bound compared with PEP.
We again observe that the VPSDP relaxation gap gets looser as the sample maximum lower bound approaches $0$ for higher values of $K$.

Note that the silver schedule derivation was exclusively for smooth unconstrained problems and, as far as we know, its convergence proof has yet to be extended to constrained problems~\cite[Section 6]{altschuler2023acceleration}.
The difficulty arises from the fact that the silver schedule cannot be studied on a per-iteration basis; rather, the entire sequence of gradient steps must be studied together with the interleaved proximal steps.
However, our VPSDP analysis offers evidence that the silver schedule can also perform well in constrained settings.

\subsection{Network utility maximization}
Consider a network of $p$ edges and $n$ flows, where each edge $i$ has capacity $d_i$ for $i=1,\dots p$, and each flow $j$ has a nonnegative flow rate $x_j$ and a maximum value $s_j$ for $j=1,\dots,n$.
Each flow routes over a specific set of edges modeled using a routing matrix 
$R \in \{0, 1\}^{p \times n}$, where $R_{ij}$ is $1$ if flow $j$ routes through edge $i$, and $0$ otherwise.
The total traffic over each edge is the sum of the flow rates over all flows that pass through it. 
Therefore, the vector of traffic for every edge is given by $Rx$, which is upper-bounded by parametric capacities $\theta$, \ie, $Rx \le \theta$.
We define the total utility as the sum of linear utilities for each flow, \ie, $w^Tx$ where $w \in \reals^n$ is a fixed vector of weights.
Our goal is to choose the flow vector $x$ that maximizes the total utility across the network, subject to capacity constraints and bounds on the flow.
This problem, often referred to as as \gls{NUM}~\cite[Section 1.3.7]{mattingley2010codegen}, can be written as the following parametric \gls{LP},
\begin{equation}\label{prob:generalLP}
    \begin{array}{ll}
		\text{minimize} & c^T x \\
		\text{subject to} & Ax \leq b(\theta),
	\end{array}
\end{equation}
where $c = -w$, $A^T = [R^T\: -I\ I] \in \reals^{m \times n}$ with $m = p + 2n$, and $b(\theta) = (\theta, 0, s) \in \reals^{p + 2n}$.
By defining the primal-dual iterate $z = (x, y) \in \reals^{n + m}$, we can write~\eqref{prob:generalLP} as the following complementarity problem
\begin{equation*}
    \ifpreprint
    \mathcal{C} \ni z \perp Mz + q \in \mathcal{C}^*, \quad \text{where}\quad  M = \begin{bmatrix}
        0 & A^T \\
        -A & 0
    \end{bmatrix} \in \reals^{(m + n)\times(m+n)}, \quad q(\theta) = \begin{bmatrix}
        c \\ b(\theta)
    \end{bmatrix} \in \reals^{m+n},
    \else
    \begin{array}{c}
        \mathcal{C} \ni z \perp Mz + q \in \mathcal{C}^*,\\
        \text{where}\quad  M = \begin{bmatrix}
            0 & A^T \\
            -A & 0
        \end{bmatrix} \in \reals^{(m + n)\times(m+n)}, \quad q(\theta) = \begin{bmatrix}
            c \\ b(\theta)
        \end{bmatrix} \in \reals^{m+n},
    \end{array} 
    \fi
\end{equation*}
and $\mathcal{C} = \reals^n \times \reals^m_+$. 
Cone $\mathcal{C}^* = \{0\}^n \times \reals^m_+$ is the dual cone to $\mathcal{C}$, defined as $\mathcal{C}^* =\{w \mid w^T u \ge 0, u \in \mathcal{C}\}$.
We solve this problem using \gls{DR} splitting~\cite{douglas1956dr,scs3}\cite[Section 3]{sambharya2022l2ws} in the form outlined in 
Algorithm \ref{alg:dr}.
We use $\Pi_\mathcal{C}(v)$ to denote the projection of a vector $v \in \reals^{n + m}$ onto cone $\mathcal{C}$ (clipping the negative values in the last $m$ elements of $v$ to $0$).

\begin{algorithm}
  \caption{\gls{DR} splitting for a general \gls{LP} problem~\eqref{prob:generalLP}.} 
	\label{alg:dr}
  \begin{algorithmic}[1]
      \State {\bf Given} problem data $(M, q(\theta))$, $K$ number of iterations, initial iterate $z^0$
      \For{$k=0, \dots, K-1$}
        \State Solve $(M+I)u^{k+1} = z^k-q(\theta)$
        \State $\tilde{u}^{k+1} = \Pi_\mathcal{C}(2u^{k+1} - z^k)$
        \State $z^{k+1} = z^k + \tilde{u}^{k+1} - u^{k+1}$
      \EndFor
      \State \textbf{return} $z^K$

  \end{algorithmic}
\end{algorithm}

In this experiment, we show how our framework can analyze warm-starting strategies by modifying the initial set $Z$.
In particular, we analyze an application-specific warm-start heuristic to set the initial iterates as follows: every initial flow rate set to the same value being the minimum, across edges, of the  ratio between the edge capacity and the number of flows passing through it~\cite[Section 1.3.7]{mattingley2010codegen}.

\paragraph{Problem setup.}
We create the routing matrix $R$ with dimensions $p=10, n=5$, which gives a fixed-point iterate of dimension $p + 3n = 25$.
For each entry $R_{ij}$, we randomly sample a Bernoulli random variable with success probability $0.5$.
We randomly sample the weight vector $w$ from a uniform distribution on $[0, 1]^n$, the saturation values $s$ from a uniform distribution on $[4, 5]^n$, and the edge capacities $\theta$ from a uniform distribution over $\Theta = \{\theta \in \reals^n \mid \|\theta - (10)\ones \|_2 \le 0.4\}$.
To obtain a representative primal-dual solution $\overline{z}$, we sample one point $\overline{\theta}$ from $\Theta$ and solve the resulting \gls{LP}.
We also use $\overline{\theta}$ to compute the heuristic-based initial iterate $\tilde{z} = (\alpha \ones, 0) \in \reals^{n + m}$ where $\alpha = \min_i \theta_i/k_i$ and $k_i$ is the number of flows that pass over edge $i$~\cite[Section 1.3.7]{mattingley2010codegen}.
We create three versions of the verification problem and compare the difference between $Z = \{0\}$ (cold start), $Z = \{\tilde{z}\}$ (heuristic start), and $Z = \{\overline{z}\}$ (warm start from representative primal-dual solution).
The PEP formulation is the same across these examples, because the problem class and algorithm do not change (see Section~\ref{sec:PEPdiff}), with the exception of the radius bounding the initial distance to optimality.

\paragraph{Results.}
We show the results in Figure~\ref{fig:num} and Table~\ref{tab:num_results}.
As implied by the sample maximum bound, the cold start has a lower residual than the heuristic start at $K=5$ for this problem setup.
Being a heuristic, there is no guarantee that the heuristic start will perform better for every NUM instance.
Our verification problem is able to pick up on this fact because of the explicit incorporation of warm-starting.

\begin{figure}[ht]
\centering
    \includegraphics[width=0.8\textwidth]{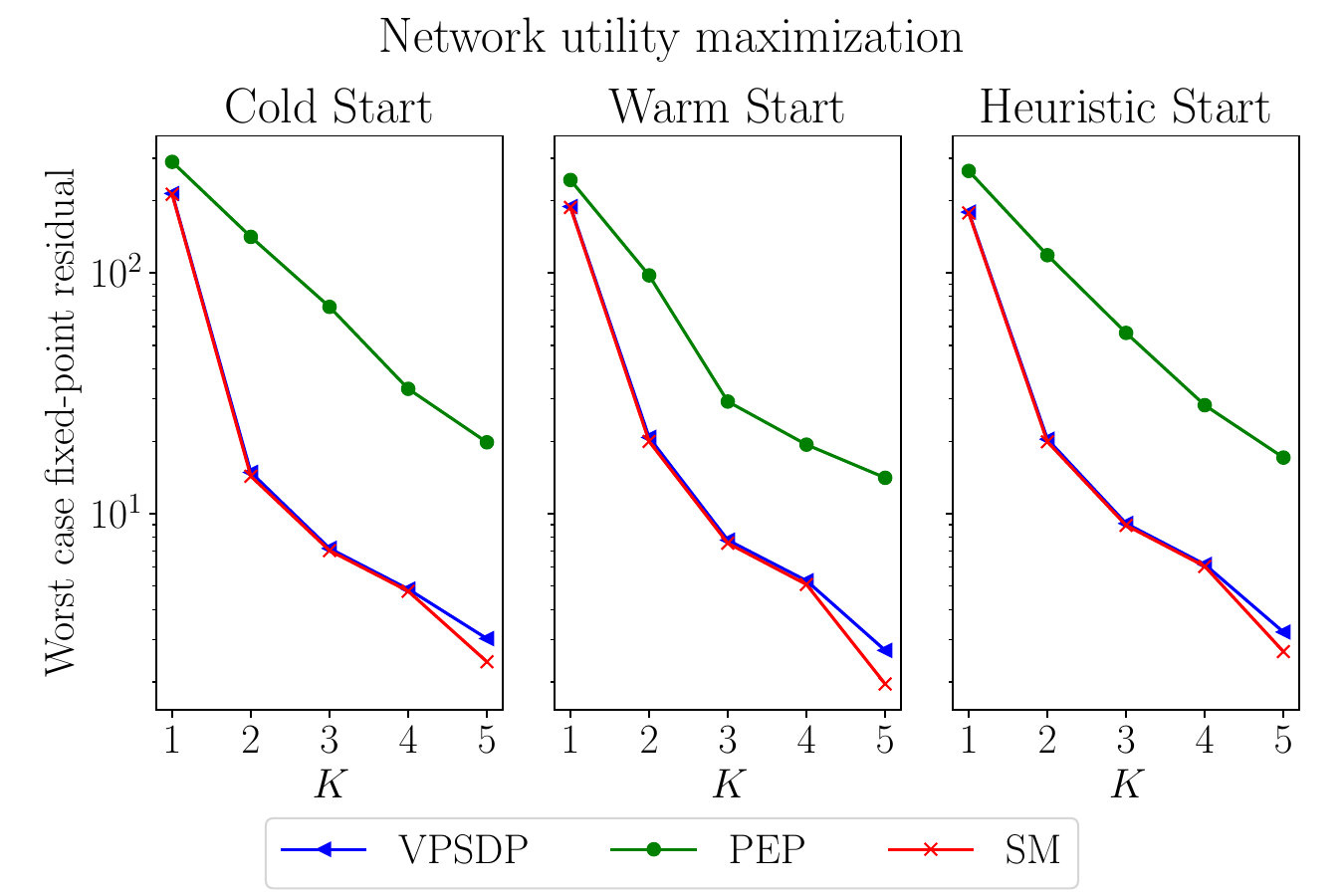}
    \caption{Results for the NUM experiment. For both PEP and VPSDP, the warm start has the tightest bound on the residual. 
    The SM bound shows that VPSDP finds a tighter bound compared to PEP. Full results are in Table~\ref{tab:num_results}.}
    \label{fig:num}
\end{figure}

\subsection{Lasso}\label{sec:lasso_exp}
To analyze the effect of Nesterov acceleration, we study the $\ell_1$-regularized least squares (Lasso)~\cite{lasso} problem,
\begin{equation}\label{prob:lasso}
    \begin{array}{ll}
		\text{minimize} & (1/2)\lVert Ax - b(\theta) \rVert^2 + \lambda \lVert x \rVert_1,
	\end{array}
\end{equation}
with $A \in \reals^{m \times n}$ and $b(\theta) \in \reals^m$.
We compare the \gls{ISTA} and \gls{FISTA} \cite[Section 4]{beck2009ista}.
The \gls{ISTA} is proximal gradient descent using the soft-thresholding function, and is detailed in Algorithm~\ref{alg:ista}.
On the other hand, the \gls{FISTA} is a version of \gls{ISTA} with carefully chosen momentum updates, and is detailed in Algorithm~\ref{alg:fista}.
\begin{algorithm}
  \caption{\gls{ISTA} for the Lasso problem~\eqref{prob:lasso}.} 
	\label{alg:ista}
  \begin{algorithmic}[1]
      \State {\bf Given} step size $t > 0$, problem data $(A, b(\theta), \lambda)$, $K$ number of iterations, initial iterate $z^0$.
      \For{$k=0, \dots, K-1$}
        \State $y^{k+1} = (I-tA^TA)z^k + tA^T b(\theta)$
\State $z^{k+1} = \max\{y^{k+1}, \lambda t\} - \max\{-y^{k+1}, \lambda t\}$
      \EndFor
      \State \textbf{return} $z^K$

  \end{algorithmic}
\end{algorithm}

\begin{algorithm}
  \caption{\gls{FISTA} for the Lasso problem~\eqref{prob:lasso}.} 
	\label{alg:fista}
  \begin{algorithmic}[1]
      \State {\bf Given} step size $t > 0$, problem data $(A, b(\theta), \lambda)$, $K$ number of iterations, initial iterates $z^0, w^0$, and parameter $\beta_0 = 1$.
      \For{$k=0, \dots, K-1$}
        \State $y^{k+1} = (I-tA^TA)w^k + tA^T b(\theta)$
\State $z^{k+1} = \max\{y^{k+1}, \lambda t\} - \max\{-y^{k+1}, \lambda t\}$
        \State $\beta_{k+1} = (1 + \sqrt{1 + 4\beta_k^2})/2$
        \State $w^{k+1} = z^{k+1} + (\beta_k - 1)/\beta_{k+1}(z^{k+1} - z^k)$
      \EndFor
      \State \textbf{return} $z^K$

  \end{algorithmic}
\end{algorithm}

\paragraph{Problem setup.}
We sample $A \in \reals^{10 \times 15}$ uniformly at random from a standard Normal distribution.
We choose $\lambda = 10$ and parameterize $b(\theta)$ as an $\ell_2$-ball of radius 0.25 centered at $(10)\ones$.
For the step size, we use $t = 0.04$ for both algorithms.
For the initial set, let $\overline{z}$ be the minimizer of $(1/2)\lVert Az - \overline{b}\rVert_2^2$, where $\overline{b}$ is randomly sampled from the parameter set.
That is, the least squares solution to a random, unregularized minimization problem.
For ISTA, we use $\{\overline{z}\}$ as the initial set for $z^0$ and for FISTA, we use $\{(\overline{z}, \overline{z})\}$ as the initial set for $(z^0, w^0)$.

\paragraph{Results.}
The results are presented in Figure~\ref{fig:lasso} and Table~\ref{tab:lasso_results}.
While FISTA has a larger worst-case fixed-point residual for the middle $K$ values, it outperforms ISTA at the largest value $K=7$.
In this experiment, the VPSDP provide particularly tight bounds, as shown by the sample maximum bound SM.
It is also worth noting that the FISTA appears to diverge with PEP, which is consistent with prior results about the convergence of the accelerated proximal gradient descent for nonsmooth objective functions \cite[Section 4.2.2]{taylor2017exactworst}.

\paragraph{Global solver comparison.}
Figure~\ref{fig:lasso_global} and Table~\ref{tab:lasso_glob_results} show the comparison between VPSDP and Gurobi nonconvex \gls{QCQP} solver.
Since all VPSDPs for Lasso were solved in under an hour, we solved the corresponding exact verification problems in Gurobi with 1\% optimality gap a 1 hour time limit.
Gurobi is able to find a good incumbent solution, compared to the lower bound SM.
We hypothesize that this is because the solver is able to exploit our precomputed lower and upper bounds in the branch and bound search.
However, we also observe that the 1 hour time limit is insufficient for the solver to solve the problem to within 1\% gap for any $K \geq 2$.
In addition, the reported gap significantly grows for larger $K$ values.
This shows that, while Gurobi is able to find a good incumbent solution, it struggles to certify optimality.

\begin{figure}
\centering
    \includegraphics[width=0.8\textwidth]{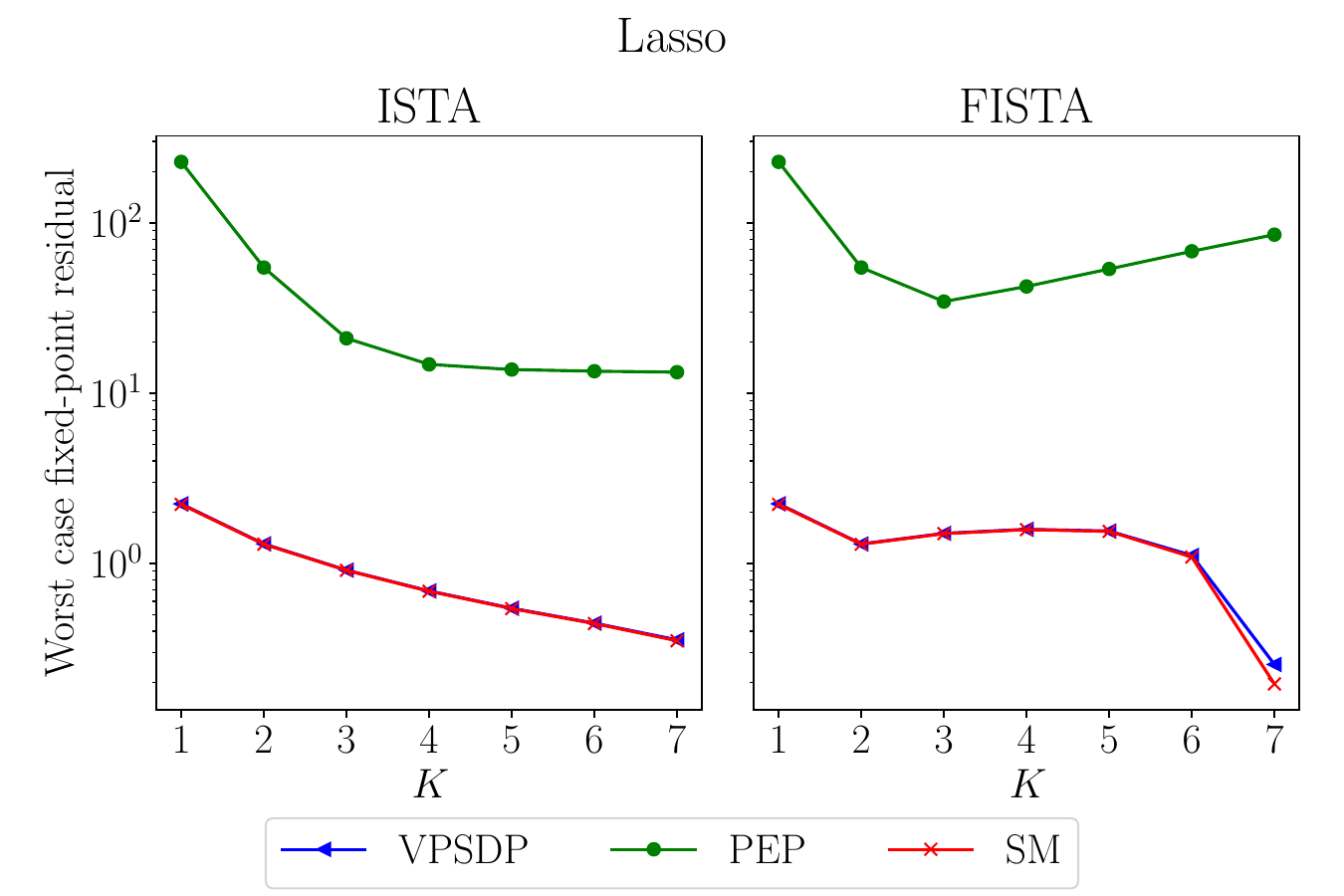}
    \caption{We show the experimental results for a Lasso problem with ISTA vs. FISTA. This shows a parametric family where FISTA outperforms ISTA with $K=7$ steps.}
    \label{fig:lasso}
\end{figure}

\begin{figure}
\centering
    \includegraphics[width=0.8\textwidth]{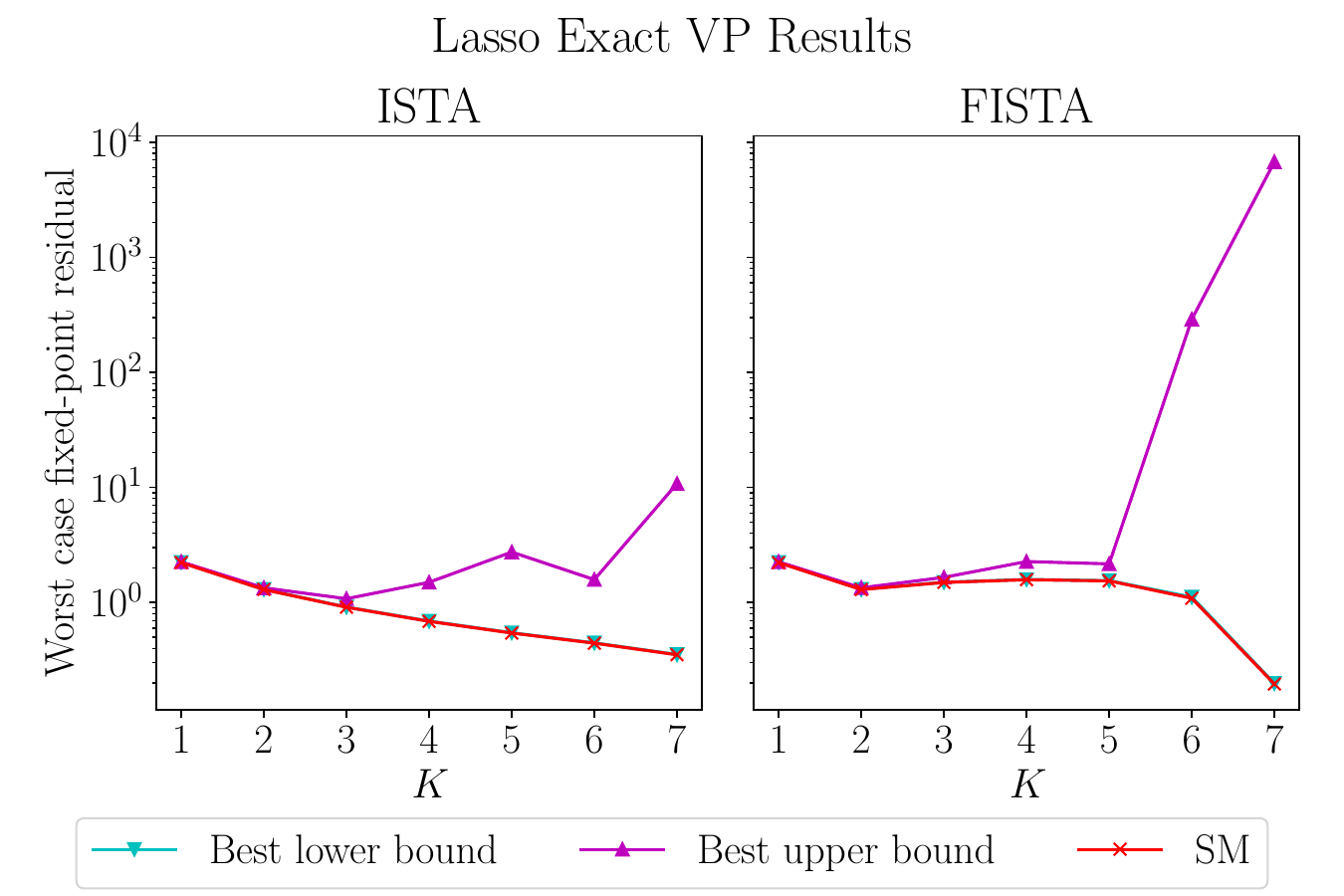}
    \caption{Lasso exact verification problem results after 1 hour solve time. Gurobi as a nonconvex \gls{QCQP} solver is able to find the incumbent solution but struggles to verify optimality and tighten the upper bounds for larger values of $K$.}
    \label{fig:lasso_global}
\end{figure}

\subsection{Optimal control}
Consider the problem of controlling an input-constrained linear dynamical system over a finite time horizon~\cite[Chapter 8.1]{borrelli2017mpc}.
The system is a vehicle, modeled as a point mass in the 2D plane, to approach the origin.
We discretize the dynamics with sampling time $h$, obtaining a finite horizon discretization indexed as $t = 1,\dots, T$.
We define the system state as $s_t = (p_t, v_t) \in \reals^4$, where $p_t \in \reals^2$ represents the position, and $v_t\in \reals^2$ the velocity at time~$t$.
We observe the output $y_t = p_t$.
The input $u_t \in \reals^2$ is the force applied to the system with mass $m$ and friction coefficient $\eta$.
We can write the linear dynamics as~\cite[Chapter 17.2]{boyd2018vmls}
\begin{equation*}
    s_{t+1} = A^{\rm dyn}s_t + B^{\rm dyn}u_t,\quad y_{t} = Cs_t,
\end{equation*}
for $t = 1,\dots, T$, with
\begin{equation*}
    A^{\rm dyn} = \begin{bmatrix}
        1 & 0 & h & 0 \\
        0 & 1 & 0 & h \\
        0 & 0 & 1-h\eta/m & 0 \\
        0 & 0 & 0 & 1-h\eta/m
    \end{bmatrix}, \quad
    B^{\rm dyn} = \begin{bmatrix}
        0 & 0 \\
        0 & 0 \\
        h/m & 0 \\
        0 & h/m
    \end{bmatrix}, \quad \text{and}\quad 
    C = \begin{bmatrix}
        1 & 0 & 0 & 0 \\
        0 & 1 & 0 & 0
    \end{bmatrix}.
\end{equation*}
Our goal is to minimize the sum of the output deviations with respect to a reference position~$y^{\rm ref} \in \reals^{2}$ and input efforts over the time horizon, \ie, $\sum_{t=1}^T \| y_t - y^{\rm ref} \|^2 + \gamma \sum_{t=1}^{T-1} \| u_t \|^2$.
We bound the maximum value of the inputs with $u^{\rm max}$, \ie, $\left\| u_t \right\|_\infty \leq u^\text{max}$, and the maximum consecutive input variations (slew rate) by $d^{\rm max}$, \ie, $\|u_t - u_{t-1} \| \leq d^{\max}$.
The optimal control problem can be written as
\begin{equation}\label{prob:optcontrol}
    \begin{array}{ll}
    \text{minimize} & \sum_{t=1}^T \lVert Cs_t - y^{\rm ref} \rVert^2 + \gamma \sum_{t=1}^{T-1} \lVert u_t \rVert^2 \\
    \text{subject to} & s_{t+1} = A^{\rm dyn}s_t + B^{\rm dyn}u_t, \quad t=1,\dots,T-1\\
    & \left\| u_t \right\|_\infty \leq u^\text{max}, \quad \left\| u_t - u_{t-1} \right\|_\infty \leq d^\text{max}, \quad t = 1,\dots,T-1\\
    & s_1 = \bar{s}.\\
    \end{array}
\end{equation}
where $\bar{s}$ is the initial state and $u_{0}$ the previous control input. Together, they represent the problem parameters $\theta = (\bar{x}, u_0)$.
Following the derivation in \cite[Chapter 8.2]{borrelli2017mpc}, we substitute out the equality constraints defining the dynamics, and obtain the following \gls{PQP} in the variables $x = (u_1, \dots, u_{T-1}) \in \reals^{2(T-1)}$,
\begin{equation}\label{prob:controlqp}
	\begin{array}{ll}
		\text{minimize} & (1/2) x^T P x\\
		\text{subject to} & l(\theta) \leq Ax \leq u(\theta).
	\end{array}
\end{equation}
The output reference cost and input costs enter into the objective through $P \in \symm^{2(T-1)}_{++}$, and the linear dynamics and input constraints enter in through the constraint matrix $A \in \reals^{4(T-1) \times 2(T-1)}$ and the bounds,~$l(\theta)$ and~$ u(\theta)$.
We solve this problem using the \gls{ADMM} in the form adopted in the \gls{OSQP} solver~\cite{stellato2020osqp} without over-relaxation.
In Algorithm~\ref{alg:osqp_admm}, we write the steps as a fixed-point iterations over iterate $z^k = (x^k, v^k)$~\cite[Section 4.1]{banjac2019infeas}, where $v^k$ implicitly represents dual and slack variables.
We use $\Pi_\mathcal{C}(v)$ to denote the projection of a vector $v \in \reals^m$ onto the set $[l(\theta), u(\theta)]$, which we outline in Table~\ref{tab:proxoper_constraints}.

\begin{algorithm}
  \caption{\gls{OSQP} algorithm without over-relaxation for a \gls{QP}~\eqref{prob:controlqp}.} 
	\label{alg:osqp_admm}
  \begin{algorithmic}[1]
      \State {\bf Given} parameters $\sigma > 0$, either $\rho > 0$ OR $\rho \in \symm_{++}$ a diagonal matrix, initial iterate $z^0 = (x^0, v^0)$, and $\mathcal{C} = [l(\theta), u(\theta)]$
      \For{$k=0,\dots,K-1$}
        \State $w^{k+1} =\Pi_{\mathcal{C}}(v^k)$
        \State Solve $(P + \sigma I + A^T \rho A)x^{k+1} = \sigma x^k + A^T \rho(2 w^{k+1} - v^k)$
        \State $v^{k+1} = \rho A x^{k+1} + (I + \rho I) v^k - (2\rho I + I)w^{k+1}$
      \EndFor
      \State \textbf{return} $z^{K} = (x^K, v^K)$.
  \end{algorithmic}
\end{algorithm}

While \gls{OSQP} adapts the step size $\rho$ throughout iterations~\cite[Section 5.2]{stellato2020osqp}, we use a fixed step size $\rho$ for simplicity.
We study two alternatives for $\rho$, scalar and diagonal.
The scalar $\rho > 0$ is a single value commonly used in \gls{ADMM}\cite{boyd2011admm}.
The diagonal $\rho$ is a diagonal matrix in $\symm^{4(T-1)}_{++}$ with different step sizes for each index.
To improve convergence~\cite[Section 5.2]{stellato2020osqp}, we set the same large value step size for the indices where $l_i(\theta) = u_i(\theta)$ (equality constraints), and the same low value for the other constraints.
For the optimal control formulation~\eqref{prob:controlqp}, the first 4 indices have $l_i(\theta) = u_i(\theta)$.

\paragraph{Problem setup.}
We choose $h=0.1, \eta=0.1, m=1, \gamma=0.2, u^\text{max}=2, d^\text{max}=0.2, T=5$.
The reference position is the origin, $y^{\rm ref} = 0$.
To generate the parametric sets, we simulate the model in a closed loop with noise for $T_\text{sim} = 25$ steps.
We start with the vehicle with initial state $\bar{s} = (5, 5, 0, 0)$ and previous input as $u_{0} = (0,0)$, \ie, initial position $(5,5)$ at rest and with no previous input force.
To generate the parameter sets, we solve the optimal control problem from the initial state, propagate the dynamics forward, and add a noise perturbation each time to compute the successive state $s_{t+1}$.
The noise is sampled from a Gaussian distribution with mean 0 and standard deviation $10^{-3}$.
We use $\theta = (s_{t+1}, u_t)$ to set the parameters for the subsequent optimal control instance.
We propagate the dynamics forward with noise for $T_\text{sim}$ steps and save all the state and input values.
After repeating the entire simulation  $N$ times, we find the smallest hypercube that contain all values of $\bar{s}$ and $u_0$, obtaining the parameter set $\Theta$.
In practice, we repeatedly solve optimal control problems with slight parameter variations $\theta$.
Therefore, a reasonable warm-start value for $x^0$ is to propagate the dynamics forward one step and shift the control input accordingly~\cite[Section IV]{marcucci2021ws_mpc}\cite{borrelli2017mpc}.
After computing  this warm-start value for every sample, we fit the smallest hypercube containing all values of $(x^0, \rho A x^0)$ to construct the initial iterate set $Z$ where iterates $z = (x, v)$ live.
In the \gls{OSQP} algorithm~\eqref{alg:osqp_admm}, we fix $\sigma = 10^{-6}$. 
First, we analyze the scalar $\rho=1$ case.
Second, we consider the diagonal $\rho$ case where $\rho_{ii} = 10$ for $i = 1,\dots,4$ and 1 otherwise.

\paragraph{Results.} The results are presented in Figure~\ref{fig:mpc} and Table~\ref{tab:mpc_results}.
The diagonal $\rho$ case has a higher fixed-point residual at $K=2$, but this is expected due to the large $\rho$ value corresponding to the equality constraints.
For higher values of $K$, the diagonal $\rho$ shows better convergence properties.
Note that PEP cannot distinguish between the scalar and diagonal $\rho$ case because of the dimension-free property discsused in Section~\ref{sec:PEPdiff}.
Since we cannot encode the specific indices of the diagonal $\rho$ matrix into a PEP SDP, we only include the PEP comparison in the scalar $\rho$ case.
This highlights an advantage of our verification problem framework; it allows us to encode strategies used in modern solvers, such as constraint-dependent step sizes, and analyze their performance.
\begin{figure}[h]
\centering
    \includegraphics[width=0.8\textwidth]{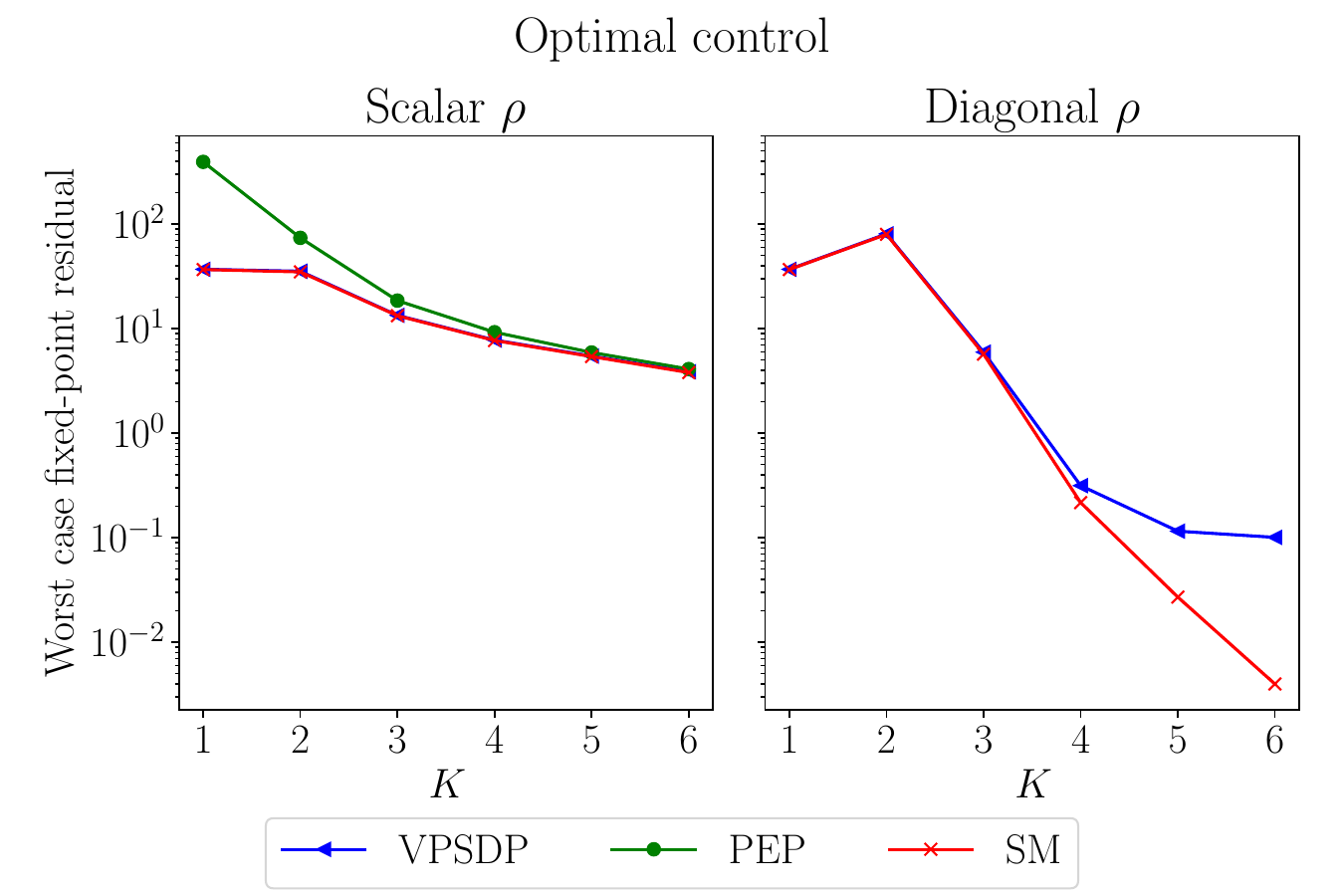}
    \caption{We show the experimental results for the optimal control problem with the OSQP algorithm for Scalar vs. Diagonal $\rho$ values. Beyond $K=2$, the Diagonal $\rho$ case outperforms the Scalar $\rho$ case.
    We cannot analyze the Diagonal $\rho$ case via PEP due to its dimension-free property.}
    \label{fig:mpc}
\end{figure}

\section{Conclusion}
We presented a framework for verifying the performance of first-order methods in parametric quadratic optimization, by evaluating the worst-case fixed-point residual after a predetermined number of iterations.
We built a collection of common proximal algorithm steps as combinations of two {\it primitive steps}: affine and element-wise maximum steps.
Our framework can explicitly quantify the effects of warm-starting, commonly used in parametric optimization, by directly representing sets for initial iterates and parameters.
We established that, in general, solving the \gls{PQP} verification problem is~\NPhard, and we constructed strong \gls{SDP} relaxations via constraint propagation and bound tightening techniques.
Using a variety of examples, we showed the flexibility of our framework, and its ability to uncover practical convergence behavior that cannot be captured using standard worst-case analysis techniques.
To the authors' knowledge, this is the first work to highlight the potential conservatism of performance estimation approaches in parametric optimization.
While our analysis was related to \glspl{PQP}, we believe that it motivates similar studies for other algorithms targeting parametric convex optimization problems. 
The main limitations of the proposed approach are that the \gls{SDP} relaxation cannot easily scale to larger \gls{PQP} instances, and that the tightness degrades as the number of algorithm steps increases.
Future works could address these issues by developing customized solution algorithms to solve large-scale semidefinite relaxations~\cite{yurtsever2019sdp}, by adopting more scalable alternatives to \gls{SDP} relaxations~\cite{ahmadidsos}, or by attempting to directly solve a nonconvex formulation of the verification problem~\cite{burer2003nonlinear}.

\ifpreprint
\section*{Acknowledgements}
\else
\begin{acknowledgements}
\fi
The authors are pleased to acknowledge that the work reported on in this paper was substantially performed using the Princeton Research Computing resources at Princeton University which is consortium of groups led by the Princeton Institute for Computational Science and Engineering (PICSciE) and Office of Information Technology's Research Computing.
\ifpreprint \else
\end{acknowledgements}
\fi

\ifpreprint \else
\section*{Statements and Declarations}
\paragraph{Competing Interests.}
The authors have no relevant financial or non-financial interests to disclose.

\fi

\newcommand{\etalchar}[1]{$^{#1}$}

\appendix

\section{Proof of Proposition \ref{prop:optimal_t}}\label{apx:opt_t_proof}
From \cite[Chapter 9]{boyd2004convex}\cite[Chapter 2]{lecturesnesterov}, for any $L$-smooth and $\mu$-strongly convex function, we have that
\begin{equation}
    \lVert z^k - z^\star \rVert \leq \tau \rVert z^{k-1} - z^\star \rVert,
\end{equation}
where $\tau = \max\{|1-t\mu|, |1-tL|\}$.
For $t \in (0, 2 /L)$, we have that $\tau \in (0, 1)$, so the gradient step is contractive.
To instead analyze the fixed-point residual, we apply the triangle inequality:
\begin{align*}
    \lVert z^k - z^{k-1} \rVert &\leq \lVert z^k - z^\star \rVert + \lVert z^{k-1} - z^\star \rVert \\
    &\leq (1+ \tau) \|z^{k-1} - z^\star \|\\
    &\leq \tau^{k-1} (1+ \tau) \|z^{0} - z^\star \|.
\end{align*}
For any $k$, the expression $\tau^{k-1}(1+\tau)$ is increasing for positive $\tau$ since the derivative is positive.
Since $\tau$ is minimized at $t=2/(\mu + L)$ \cite[Equation 1.2.27]{lecturesnesterov} and $\tau > 0$, the right hand side coefficient is minimized for the same $t$ value as well.

\section{Proofs for Section~\ref{sec:cvx_relaxation}}

\subsection{Proof of Theorem~\ref{thm:nphard}}
\label{apx:nphard}
We show problem~\eqref{prob:general_verifyprob} is~\NPhard~via a reduction from the 0-1 integer programming feasibility problem.
Specifically, deciding if there exists a feasible solution to the following integer program:
\begin{equation}\label{prob:01_IP}
\begin{array}{ll}
    \mbox{minimize} & 0 \\
    \mbox{subject to} & Ax = b,\\
    & x \in \{0, 1\}^n,
\end{array}
\end{equation}
is an~\NPhard~problem \cite[Section 4]{karp1972NP}.
Consider the following related optimization problem:
\begin{equation}\label{prob:IP_QPreduction}
\begin{array}{ll}
    \mbox{maximize} & \lVert x - (1/2)\ones \rVert^2 \\
    \mbox{subject to} & Ax = b,\\
    & 0 \leq x \leq 1.
\end{array}
\end{equation}
Note that the integrality constraints from~\eqref{prob:01_IP} were relaxed to be continuous in~\eqref{prob:IP_QPreduction}.
We have $ \lVert x - (1/2)\ones \rVert^2 = \sum_{i=1}^n (x_i - 1/2)^2$, and on $0 \leq x_i \leq 1$, $(x_i - 1/2)^2$ is maximized at $x_i = 0$ or $x_i = 1$ with value $1/4$.
Since $(x_i - 1/2)^2 \geq 0$, the optimal value of\eqref{prob:IP_QPreduction} is $n/4$ if and only if~\eqref{prob:01_IP} is feasible.
So, it is~\NPhard~to determine if the optimal value of~\eqref{prob:IP_QPreduction} is $n/4$.
We can rewrite~\eqref{prob:IP_QPreduction} as
\begin{equation}\label{prob:QPreduction_verifyprob}
\begin{array}{ll}
    \mbox{maximize} & \lVert y^1 - y^0 \rVert^2 \\
\mbox{subject to} & y^1 = x^0 - (1/2)q(\theta),\\
    & y^0 = 0,\\
    & q(\theta) = \ones,\\
    & Ax^0 = b,\\
    & 0 \leq x^0 \leq 1.
\end{array}
\end{equation}
Every constraint in~\eqref{prob:QPreduction_verifyprob} is either an affine step, a box constraint, or a polyhedral constraint.
So, we have reduced the 0-1 integer feasibility problem to deciding the maximum value of a verification problem, which proves that finding the optimal value of the verification problem is~\NPhard.

\subsection{Proof of Proposition~\ref{prop:obj_relax}}\label{apx:proof_obj_relax}

We first introduce a helper lemma about traces and matrix products.
\begin{lemma}[trace of positive semidefinite matrix product]\label{lemma:psd_matrix_prod_trace}
    If $A, B \in \symm^{d \times d}_+$ are positive semidefinite, then $\trace(AB) \geq 0$.
\end{lemma}
\begin{proof}
Since $B$ is positive semidefinite, it has a square root, \ie~a positive semidefinite matrix $H \in \symm_+^{d\times d}$ such that $B = HH^T$.
Applying the cyclic property of the trace gives
\begin{equation*}
    \trace(AB) = \trace(A HH^T) = \trace(H^T A H).
\end{equation*}
For any $x$, since $A$ is also positive semidefinite,
\begin{equation*}
   x^T H^T A H x = (Hx)^T A (Hx) \geq 0.
\end{equation*}
This proves $H^T A H$ is positive semidefinite, therefore $\trace(H^T A H) \geq 0$.
\end{proof}

We can now proceed in proving Proposition~\ref{prop:obj_relax}. By rewriting $f(x,y)$ and applying the cyclic property of the trace, we have:
\begin{equation*}
    \|x - y \|^2 = \begin{bmatrix}x \\ y\end{bmatrix}^T \begin{bmatrix}
        I & -I \\ -I & I
    \end{bmatrix} \begin{bmatrix}x \\ y\end{bmatrix} = \Tr\left(\begin{bmatrix}
        I & -I \\ -I & I
    \end{bmatrix} \begin{bmatrix}x \\ y\end{bmatrix}\begin{bmatrix}x \\ y\end{bmatrix}^T\right).
\end{equation*}
We choose 
\begin{equation*}
    M = \begin{bmatrix}x \\ y\end{bmatrix}\begin{bmatrix}x \\ y\end{bmatrix}^T,
\end{equation*}
relax the equality to inequality, and apply the Schur complement, similar to the derivation in~\eqref{eqn:matvar_psdrelaxation}.

Since the squared norm is a convex function and the difference $x-y$ is affine, the entire expression is convex and, therefore, the coefficient matrix inside the trace is positive semidefinite.
This allows us to apply Lemma~\ref{lemma:psd_matrix_prod_trace} and finish the proof.

\subsection{Proof of Proposition~\ref{prop:affine_relax}}\label{apx:proof_affine_relax}

Since $D$ is assumed to be invertible, $F$ is nonempty for any $A, B, D$.
Given \reviewChanges{$(y, x, q, M_1, M_2) \in F$,}
observe that
\begin{equation*}
    (Dy)(Dy)^T = (Ax + Bq)(Ax + Bq)^T \implies Dyy^TD^T = \begin{bmatrix} A & B \end{bmatrix} 
\begin{bmatrix}
    x \\ q
    \end{bmatrix} 
     \begin{bmatrix}
    x \\ q
    \end{bmatrix}^T
\begin{bmatrix}
    A^T \\
    B^T
\end{bmatrix}.
\end{equation*}
\reviewChanges{Given the constraints for inclusion in $F$, we have}
\begin{equation*}
    M_1 = yy^T, \quad M_2 = 
\begin{bmatrix}
    x \\ q
    \end{bmatrix} 
     \begin{bmatrix}
    x \\ q
    \end{bmatrix}^T,
\end{equation*}
\reviewChanges{We finish the proof by} relaxing the equalities to inequalities, and applying the Schur complement to each.

\subsection{Proof of Proposition~\ref{prop:max_relax}}\label{apx:proof_max_relax}

Similarly to our proof of Proposition~\ref{prop:affine_relax}, given \reviewChanges{$(y, x, l, M) \in F$}, 
recall from Section~\ref{sec:fixedpt_core} that
\begin{equation*}
    y = \max\{x, l\} \iff y \geq x,\quad y \geq l,\quad (y - l)^T(y-x) = 0.
\end{equation*}
The complementarity constraint requires cyclic property of the trace:
    \begin{align*}
    (y - l)^T(y-x)
&=
    \begin{bmatrix}
        y \\ x \\ l
    \end{bmatrix}^T
    \begin{bmatrix}
        I & -I/2 & -I/2 \\
        -I/2 & 0 & I/2 \\
        -I/2 & I/2 & 0
    \end{bmatrix}
    \begin{bmatrix}
        y \\ x \\ l
    \end{bmatrix}\\
    &=
    \trace \left( \begin{bmatrix}
        I & -I/2 & -I/2 \\
        -I/2 & 0 & I/2 \\
        -I/2 & I/2 & 0
    \end{bmatrix}
\begin{bmatrix}
        y \\ x \\ l
    \end{bmatrix}
    \begin{bmatrix}
        y \\ x \\ l
    \end{bmatrix}^T
    \right) = 0.
\end{align*}
\reviewChanges{
    Given the constraints for inclusion in $F$, we have
}
\begin{equation*}
    M = \begin{bmatrix}
        y \\ x \\ l
    \end{bmatrix}
    \begin{bmatrix}
        y \\ x \\ l
    \end{bmatrix}^T,
\end{equation*}
and we finish the proof by relaxing the equality to inequality and applying the Schur complement.

\subsection{Proof of Proposition~\ref{prop:hypercube_relax}}\label{apx:proof_hypercube_relax}

The bounds $l \leq z \leq u$ can be rewritten as $u - z \geq 0,\; z -l \geq 0$.
By cross-multiplying the lower and upper bound inequalities, we derive:
\begin{equation*}
    (u - z)(z - l)^T = uz^T - ul^T - zz^T + zl^T \geq 0.
\end{equation*}
\reviewChanges{
    Additionally, we multiply the inequalities to themselves to derive:
    \begin{align*}
        (u - z)(u - z)^T = uu^T - uz^T - zu^T + zz^T &\geq 0\\
        (z - l)(z - l)^T = zz^T - zl^T - lz^T + ll^T &\geq 0.
    \end{align*}
    We have $M = zz^T$, relax this matrix equality to inequalites, and apply the Schur complement to finish the proof.
    Similarly to the RLT discussion in Section~\ref{subsec:tightening_sdp}, these sets of matrix inequalities imply the original bounds $l \le z \le u$~\cite[Proposition 1]{sherali1995reformulationconvexification}.
}

\subsection{Proof of Proposition~\ref{prop:polyhedra_relax}}\label{apx:proof_polyhedra_relax}
We relax the constraints in a similar way to our proof in Proposition~\ref{prop:affine_relax}.
Specifically, \reviewChanges{given $(z, s, M) \in F$,} we multiply the equality constraints with themselves:
\begin{equation*}
    (Az + s)(Az + s)^T = \begin{bmatrix}
    A & I
\end{bmatrix} 
\begin{bmatrix}
    z \\
    s
\end{bmatrix}
\begin{bmatrix}
    z \\
    s
\end{bmatrix}^T
\begin{bmatrix}
    A \\ I
\end{bmatrix}
= bb^T.
\end{equation*}
\reviewChanges{Given the constraints of $F$, we have}
\begin{equation*}
    M = \begin{bmatrix}
    z \\
    s
\end{bmatrix}
\begin{bmatrix}
    z \\
    s
\end{bmatrix}^T.
\end{equation*}
We relax the equalities to inequalities and apply the Schur complement \reviewChanges{to finish the proof}.

\subsection{Proof of Proposition~\ref{prop:l2ball_relax}}\label{apx:proof_l2_ball_relax}
\reviewChanges{
Given $(z, M) \in F$, squaring the norm constraint in $F$ gives
\begin{equation*}
    z^Tz - 2z^Tc + c^Tc \leq r^2.
\end{equation*}
By setting $M = zz^T$ and applying the cyclic property of the trace, we have
\begin{equation*}
    z^Tz - 2z^Tc + c^Tc = \trace(M) - 2z^Tc + c^Tc.
\end{equation*}
By relaxing the matrix equality to inequality, we have
\begin{equation*}
    z^Tz - 2z^Tc + c^Tc \le \trace(M) - 2z^Tc + c^Tc \le r^2,
\end{equation*}
and we apply the Schur complement to finish the proof.
}

\section{\reviewChanges{Complete \gls{SDP} relaxation for NNLS example}}\label{apx:sdp_coupling}
\reviewChanges{
In this Section, we show the VPSDP for $K=2$ using the \gls{NNLS} example and setup in Section~\ref{subsec:nnls} with initial iterate set $Z = \{0 \}$ and parameter set $\Theta$ as an $\ell_2$-ball of radius $r$ centered at $c$.
We use this example to expand on the idea of reusing matrix variables briefly introduced in Section~\ref{subsec:tightening_sdp}.
Letting $A_t = I - tA^TA$, $B_t = tA^T$, the VP for Algorithm~\ref{alg:proj_gd} becomes
\begin{equation}\label{prob:nnls_vp}
    \begin{array}{ll}
        \text{minimize} & \| z^2 - z^1 \|^2 \\
        \text{subject to} & y^1 = A_t z^0 + B_t b(\theta) \\
        & z^1 = (y^1)_+ \\
        & y^2 = A_t z^1 + B_t b(\theta) \\
        & z^2 = (y^2)_+ \\
        & z^0 = 0, \quad \| b(\theta) - c \| \le r.
    \end{array}
\end{equation}
Applying Propositions~\ref{prop:affine_relax} and~\ref{prop:max_relax} in isolation give the following matrix variables:
\begin{equation*}
    \begin{array}{l}
        \begin{bmatrix}
            M_1 & y^{1} \\
            (y^{1})^T & 1
        \end{bmatrix},
        \quad  
        \begin{bmatrix}
        \multicolumn{2}{c}{\multirow{2}{*}{$M_2$}}& z^0\\
        \multicolumn{2}{c}{}& b(\theta)\\
        (z^0)^T & b(\theta)^T & 1
        \end{bmatrix}, 
        \quad
        \begin{bmatrix}
            \multicolumn{2}{c}{\multirow{2}{*}{$M_3$}}& z^1\\
            \multicolumn{2}{c}{}& y^1\\
            (z^1)^T & (y^1)^T & 1
        \end{bmatrix}, \\[2em]
        \begin{bmatrix}
            M_4 & y^{2} \\
            (y^{2})^T & 1
        \end{bmatrix},
        \quad  
        \begin{bmatrix}
        \multicolumn{2}{c}{\multirow{2}{*}{$M_5$}}& z^1\\
        \multicolumn{2}{c}{}& b(\theta)\\
        (z^1)^T & b(\theta)^T & 1
        \end{bmatrix}, 
        \quad
        \begin{bmatrix}
            \multicolumn{2}{c}{\multirow{2}{*}{$M_6$}}& z^2\\
            \multicolumn{2}{c}{}& y^2\\
            (z^2)^T & (y^2)^T & 1
            \end{bmatrix}.
    \end{array}
\end{equation*}
While these matrix variables arise independently, the dimension of the VPSDP can be reduced by considering them together.
For example, the matrix variables involving $M_2$ and $M_5$ both have a component involving $b(\theta)$, so the bottom right sub-block of both variables correspond to the relaxed matrix variable for $b(\theta)b(\theta)^T$.
In order to maintain equality between such shared components, we use a consensus variable formulation from the \gls{SDP} chordal sparsity literature \cite{zheng2021chordal}, \cite[Section III]{zheng2017chordaladmm}.
}

\reviewChanges{
For notation, let $M_{x,y}$ correspond to the matrix variable that replaces the outer product variable $xy^T$.
For example, the consensus variable formulation means we create a single matrix variable $M_{b(\theta), b(\theta)}$ and reuse it in all necessary positive semidefinite constraints.
Letting $C_t = \begin{bmatrix}
    A_t & B_t \end{bmatrix}$, the full VPSDP corresponding to~\eqref{prob:nnls_vp} becomes
\begin{equation}
    \begin{array}{ll}
        \text{minimize} & \trace\left(\begin{bmatrix}
            I & -I \\ -I & I
        \end{bmatrix}\begin{bmatrix}
            M_{z^2, z^2} & M_{z^2, z^1} \\ M^T_{z^2, z^1} & M_{z^1, z^1}
        \end{bmatrix}\right) \\
        \text{subject to}
        & \begin{bmatrix}
            M_{z^2, z^2} & M_{z^2, z^1} \\ M^T_{z^2, z^1} & M_{z^1, z^1}
        \end{bmatrix} \succeq 0, \\
        & y^1 = A_t z^0 + B_t b(\theta), \\
        & M_{y^1, y^1} = C_t \begin{bmatrix}
            M_{z^0, z^0} & M_{z^0, b(\theta)} \\ M^T_{z^0, b(\theta)} & M_{b(\theta), b(\theta)}
        \end{bmatrix} C_t^T,\\
        & \begin{bmatrix}
            M_{z^0, z^0} & M_{z^0, b(\theta)} & z^0 \\ M^T_{z^0, b(\theta)} & M_{b(\theta), b(\theta)} & b(\theta) \\ (z^0)^T & b(\theta)^T & 1
        \end{bmatrix} \succeq 0, \\
        & z^1 \geq 0, \quad z^1 \geq y^1,\\
        & \trace\left(\begin{bmatrix}
            I & -I/2 \\ -I/2 & 0
        \end{bmatrix} \begin{bmatrix}
            M_{z^1, z^1} & M_{z^1, y^1} \\ M^T_{z^1, y^1} & M_{y^1, y^1}
        \end{bmatrix} \right) = 0, \\
        & \begin{bmatrix}
            M_{z^1, z^1} & M_{z^1, y^1} & z^1 \\ M^T_{z^1, y^1} & M_{y^1, y^1} & y^1 \\ (z^1)^T & (y^1)^T & 1
        \end{bmatrix} \succeq 0, \\
        & y^2 = A_t z^1 + B_t b(\theta), \\
        & M_{y^2, y^2} = C_t \begin{bmatrix}
            M_{z^1, z^1} & M_{z^1, b(\theta)} \\ M^T_{z^1, b(\theta)} & M_{b(\theta), b(\theta)}
        \end{bmatrix} C_t^T,\\
        & \begin{bmatrix}
            M_{z^1, z^1} & M_{z^1, b(\theta)} & z^1 \\ M^T_{z^1, b(\theta)} & M_{b(\theta), b(\theta)} & b(\theta) \\ (z^1)^T & b(\theta)^T & 1
        \end{bmatrix} \succeq 0, \\
        & z^2 \geq 0, \quad z^2 \geq y^2,\\
        & \trace\left(\begin{bmatrix}
            I & -I/2 \\ -I/2 & 0
        \end{bmatrix} \begin{bmatrix}
            M_{z^2, z^2} & M_{z^2, y^2} \\ M^T_{z^2, y^2} & M_{y^2, y^2}
        \end{bmatrix} \right) = 0, \\
        & \begin{bmatrix}
            M_{z^2, z^2} & M_{z^2, y^2} & z^2 \\ M^T_{z^2, y^2} & M_{y^2, y^2} & y^2 \\ (z^2)^T & (y^2)^T & 1
        \end{bmatrix} \succeq 0,\\
        & z^0 = 0, \quad M_{z^0, z^0} = 0, \\
        & \trace\left(M_{b(\theta), b(\theta)}\right) - 2 b(\theta)^T c + c^T c \le r^2.
    \end{array}
\end{equation}
Another benefit of this consensus variable formulation is that we can reduce the number of positive semidefinite constraints.
Instead of directly including the positive semidefinite constraint from Proposition~\ref{prop:l2ball_relax}, it is implied from the larger positive semidefinite constraints.
}

\section{Numerical result tables}\label{apx:fullresulttables}
In this section we provide full data tables for every experiment in Section~\ref{sec:experiments}.

\begin{table}[h]
      \caption{Nonnegative least squares results, nonstrongly convex case, and step size fixed across $K$. Corresponds to Figure~\ref{fig:nonstrong_NNLS_gridt} (first part).}
      \label{tab:nnls_nonstrong_fixed_pt1}
      \begin{center}
      \adjustbox{max width=\textwidth}{
        \begin{tabular}{llllllll}
          \toprule
           $t$ & $K$ & $\text{VPSDP}$ & $\text{SM}$ & $\text{PEP}$ 
           & 
           \begin{tabular}{@{}c@{}}VPSDP \\ solve time (s)\end{tabular}
           & 
           $\frac{\text{PEP}}{\text{VPSDP}}$
           & 
           $\frac{\text{PEP}}{\text{SM}}$
           \\
          \midrule
          \csvreader[
            separator=comma,
            late after line=\\,
            late after last line=\\\bottomrule,]
            {NNLS_nonstrong_rounded_pt1.bst}{}
            {\ifnum\thecsvrow=11 \hline\fi 
            \ifnum\thecsvrow=21 \hline\fi
            \csvlinetotablerow}
        \end{tabular}}
      \end{center}
\end{table}

\begin{table}[h]
      \caption{Nonnegative least squares results, nonstrongly convex case, and step size fixed across $K$. Corresponds to Figure~\ref{fig:nonstrong_NNLS_gridt} (second part).}
      \label{tab:nnls_nonstrong_fixed_pt2}
      \begin{center}
      \adjustbox{max width=\textwidth}{
        \begin{tabular}{llllllll}
          \toprule
           $t$ & $K$ & $\text{VPSDP}$ & $\text{SM}$ & $\text{PEP}$ 
           & 
           \begin{tabular}{@{}c@{}}VPSDP \\ solve time (s)\end{tabular}
           & 
           $\frac{\text{PEP}}{\text{VPSDP}}$
           & 
           $\frac{\text{PEP}}{\text{SM}}$
           \\
          \midrule
          \csvreader[
            separator=comma,
            late after line=\\,
            late after last line=\\\bottomrule,]
            {NNLS_nonstrong_rounded_pt2.bst}{}
            {\ifnum\thecsvrow=11 \hline\fi 
\csvlinetotablerow}
        \end{tabular}}
      \end{center}
\end{table}

\begin{table}[h]
      \caption{Nonnegative least squares results, nonstrongly convex case, and silver step size schedule. Corresponds to Figure~\ref{fig:NNLS_silver_nonstrongcvx}.}
      \label{tab:nnls_silver_nonstrongcvx}
      \begin{center}
      \adjustbox{max width=\textwidth}{
        \begin{tabular}{lllllllll}
          \toprule
           Schedule & $K$ & $t$ & $\text{VPSDP}$ & $\text{SM}$
           & $\text{PEP}$ 
           & 
           \begin{tabular}{@{}c@{}}VPSDP \\ solve time (s)\end{tabular}
           & 
           $\frac{\text{PEP}}{\text{VPSDP}}$
           & 
           $\frac{\text{PEP}}{\text{SM}}$
           \\
          \midrule
          \csvreader[
            separator=comma,
            late after line=\\,
            late after last line=\\\bottomrule,]
            {nonstrongsilver_ratio_rounded.bst}{}
            {\ifnum\thecsvrow=8 \hline\fi
            \csvlinetotablerow}
        \end{tabular}}
      \end{center}
\end{table}

\begin{table}[h]
      \caption{Nonnegative least squares results, strongly convex case, and step size fixed across $K$.Corresponds to Figure~\ref{fig:strong_NNLS_gridt} (first part).}
      \label{tab:nnls_fixed_pt1}
      \begin{center}
      \adjustbox{max width=\textwidth}{
        \begin{tabular}{llllllll}
          \toprule
           $t$ & $K$ & $\text{VPSDP}$ & $\text{SM}$ & $\text{PEP}$ 
           & 
           \begin{tabular}{@{}c@{}}VPSDP \\ solve time (s)\end{tabular}
           & 
           $\frac{\text{PEP}}{\text{VPSDP}}$
           & 
           $\frac{\text{PEP}}{\text{SM}}$
           \\
          \midrule
          \csvreader[
            separator=comma,
            late after line=\\,
            late after last line=\\\bottomrule,]
            {NNLS_ratio_rounded_pt1.bst}{}
            {\ifnum\thecsvrow=11 \hline\fi 
            \ifnum\thecsvrow=21 \hline\fi
            \csvlinetotablerow}
        \end{tabular}}
      \end{center}
\end{table}

\begin{table}[h]
      \caption{Nonnegative least squares results, strongly convex case, and step size fixed across $K$.Corresponds to Figure~\ref{fig:strong_NNLS_gridt} (second part).}
      \label{tab:nnls_fixed_pt2}
      \begin{center}
      \adjustbox{max width=\textwidth}{
        \begin{tabular}{llllllll}
          \toprule
           $t$ & $K$ & $\text{VPSDP}$ & $\text{SM}$ & $\text{PEP}$ 
           & 
           \begin{tabular}{@{}c@{}}VPSDP \\ solve time (s)\end{tabular}
           & 
           $\frac{\text{PEP}}{\text{VPSDP}}$
           & 
           $\frac{\text{PEP}}{\text{SM}}$
           \\
          \midrule
          \csvreader[
            separator=comma,
            late after line=\\,
            late after last line=\\\bottomrule,]
            {NNLS_ratio_rounded_pt2.bst}{}
            {\ifnum\thecsvrow=11 \hline\fi 
            \ifnum\thecsvrow=21 \hline\fi
            \csvlinetotablerow}
        \end{tabular}}
      \end{center}
\end{table}

\begin{table}[h]
      \caption{Nonnegative least squares results, strongly convex case, and silver step size schedule. Corresponds to Figure~\ref{fig:NNLS_silver_strongcvx}.}
      \label{tab:nnls_silver_strongcvx}
      \begin{center}
      \adjustbox{max width=\textwidth}{
        \begin{tabular}{lllllllll}
          \toprule
           Schedule & $K$ & $t$ & $\text{VPSDP}$ & $\text{SM}$
           & $\text{PEP}$ 
           & 
           \begin{tabular}{@{}c@{}}VPSDP \\ solve time (s)\end{tabular}
           & 
           $\frac{\text{PEP}}{\text{VPSDP}}$
           & 
           $\frac{\text{PEP}}{\text{SM}}$
           \\
          \midrule
          \csvreader[
            separator=comma,
            late after line=\\,
            late after last line=\\\bottomrule,]
            {strongsilver_ratio_rounded.bst}{}
            {\ifnum\thecsvrow=8 \hline\fi 
            \csvlinetotablerow}
        \end{tabular}}
      \end{center}
\end{table}

\begin{table}[h]
      \caption{Results for the network utility maximization experiment in Figure~\ref{fig:num}.}
      \label{tab:num_results}
      \begin{center}
      \adjustbox{max width=\textwidth}{
        \begin{tabular}{llllllll}
          \toprule
           Init. & $K$ & $\text{VPSDP}$ & $\text{SM}$ & $\text{PEP}$ 
           & 
           \begin{tabular}{@{}c@{}}VPSDP\\solve time (s)\end{tabular}
           & 
           $\frac{\text{PEP}}{\text{VPSDP}}$
           & 
           $\frac{\text{PEP}}{\text{SM}}$
           \\
          \midrule
          \csvreader[
            separator=comma,
            late after line=\\,
            late after last line=\\\bottomrule,]
            {NUM_ratio_rounded.bst}{}
            {\ifnum\thecsvrow=6 \hline\fi 
            \ifnum\thecsvrow=11 \hline\fi
            \csvlinetotablerow}
        \end{tabular}}
      \end{center}
\end{table}

\begin{table}[h]
      \caption{Results for the Lasso experiment in Figure~\ref{fig:lasso}.}
      \label{tab:lasso_results}
      \begin{center}
      \adjustbox{max width=\textwidth}{
        \begin{tabular}{llllllll}
          \toprule
           Alg. & $K$ & $\text{VPSDP}$ & $\text{SM}$ & $\text{PEP}$ 
           & 
           \begin{tabular}{@{}c@{}}VPSDP \\ solve time (s)\end{tabular}
           & 
           $\frac{\text{PEP}}{\text{VPSDP}}$
           & 
           $\frac{\text{PEP}}{\text{SM}}$
           \\
          \midrule
          \csvreader[
            separator=comma,
            late after line=\\,
            late after last line=\\\bottomrule,]
            {ISTA_ratio_rounded.bst}{}
{\ifnum\thecsvrow=8 \hline\fi \csvlinetotablerow}
        \end{tabular}}
      \end{center}
\end{table}

\begin{table}[h]
      \caption{Exact verification problem results for Lasso with a 1 hour time limit with a target optimality gap of 1\%. Corresponds to Figure~\ref{fig:lasso_global}.}
      \label{tab:lasso_glob_results}
      \begin{center}
      \adjustbox{max width=\textwidth}{
        \begin{tabular}{llllll}
          \toprule
           Alg. & $K$ & 
           \begin{tabular}{@{}c@{}}Best \\ lower bound \end{tabular}
           &  \begin{tabular}{@{}c@{}}Best \\ upper bound \end{tabular}
           & Optimality gap (\%) 
           & 
           \begin{tabular}{@{}c@{}}Gurobi \\ solve time (s)\end{tabular}
           \\
          \midrule
          \csvreader[
            separator=comma,
            late after line=\\,
            late after last line=\\\bottomrule,]
            {ISTA_glob_roundsci.bst}{}
{\ifnum\thecsvrow=8 \hline\fi \csvlinetotablerow}
        \end{tabular}}
      \end{center}
\end{table}

\begin{table}[ht]
      \caption{Results for the optimal control problem in Figure~\ref{fig:mpc}.}
      \label{tab:mpc_results}
      \begin{center}
      \adjustbox{max width=\textwidth}{
        \begin{tabular}{llllllll}
          \toprule
           $\rho$ & $K$ & $\text{VPSDP}$ & $\text{SM}$ & $\text{PEP}$ 
           & 
           \begin{tabular}{@{}c@{}}VPSDP \\ solve time (s)\end{tabular}
           & 
           $\frac{\text{PEP}}{\text{VPSDP}}$
           & 
           $\frac{\text{PEP}}{\text{SM}}$
           \\
          \midrule
          \csvreader[
            separator=comma,
            late after line=\\,
            late after last line=\\\bottomrule,]
            {MPC_ratio_rounded.bst}{}
{\ifnum\thecsvrow=7 \hline\fi \csvlinetotablerow}
        \end{tabular}}
      \end{center}
\end{table}

\end{document}